\documentclass[11pt]{amsart}
\usepackage{amsmath,amssymb,amsthm,mathtools}
\usepackage{booktabs}
\usepackage{placeins}
\usepackage[margin=1.05in]{geometry}
\usepackage[dvipsnames]{xcolor}
\usepackage[colorlinks=true,linkcolor=RoyalBlue,citecolor=RoyalBlue,urlcolor=RoyalBlue]{hyperref}
\makeatletter
\renewcommand{\eqref}[1]{\textup{\hyperref[#1]{\tagform@{\ref*{#1}}}}}
\makeatother

\newtheorem{theorem}{Theorem}[section]
\newtheorem{lemma}[theorem]{Lemma}
\newtheorem{proposition}[theorem]{Proposition}

\theoremstyle{definition}
\newtheorem{assumption}[theorem]{Assumption}
\newtheorem{definition}[theorem]{Definition}
\theoremstyle{remark}
\newtheorem{remark}[theorem]{Remark}
\numberwithin{equation}{section}
\newcommand{\E}{\mathbb{E}}
\newcommand{\Prob}{\mathbb{P}}
\newcommand{\R}{\mathbb{R}}
\newcommand{\Z}{\mathbb{Z}}
\newcommand{\one}{\mathbf{1}}
\newcommand{\bH}{\mathbf{H}}
\newcommand{\tbH}{\widetilde{\mathbf{H}}}
\newcommand{\hbH}{\widehat{\mathbf{H}}}
\newcommand{\Var}{\operatorname{Var}}

\newcommand{\tr}{\operatorname{tr}}
\newcommand{\Hmax}{H_{\max}}
\newcommand{\tHmax}{\widetilde H_{\max}}
\newcommand{\dto}{\xrightarrow{\;d\;}}

\newcommand{\VIF}{\mathcal{V}}

\begin{document}

\title[Plug-in inference for multi-mixed fBm with drift]{Plug-in inference for multi-mixed fractional Brownian motion with drift and unknown Hurst indices}

\author{Afrah Al-Harby}
\address{Department of Mathematics, College of Science, Imam Abdulrahman Bin Faisal University,\newline P.O.~Box 1982, Dammam 31441, Saudi Arabia}
\email{2250500238@iau.edu.sa}

\author{Ezzedine Mliki}
\address{Department of Mathematics, College of Science, Imam Abdulrahman Bin Faisal University,\newline P.O.~Box 1982, Dammam 31441, Saudi Arabia}
\address{Basic and Applied Scientific Research Center, Imam Abdulrahman Bin Faisal University,\newline P.O.~Box 1982, Dammam 31441, Saudi Arabia}
\email{ermliki@iau.edu.sa}
\thanks{Corresponding author: Ezzedine Mliki (ermliki@iau.edu.sa).}


\begin{abstract}
\hyphenpenalty=10000
In this paper we study inference for the drift $\theta$ and the scale $\alpha$ of the process ${Z(t)=\theta t+\alpha\sum_{r=1}^{m}B^{H_r}(t)}$, observed at $N$ equidistant times, when the Hurst indices are unknown. If a wrong Hurst vector is used, the coverage of the Student interval for $\theta$ tends to $0$ when the largest index is underestimated and to $1$ when it is overestimated. If the indices are estimated by maximum likelihood, the plug-in drift estimator is asymptotically equivalent to the oracle one, and the plug-in Student interval is asymptotically valid. The chi-square interval for $\alpha^2$ is not: its coverage tends to $2\Phi(z_{1-\beta/2}/\sqrt{\mathcal V})-1$, where $\mathcal V\ge1$ is an explicit inflation factor. We give corrected intervals, show that the Hurst indices are weakly identified when they cluster, and confirm the results by simulation.
\end{abstract}

\maketitle

\section{Introduction}\label{sec:intro}

A fractional Brownian motion (fBm) $B^H=\{B^H(t),t\ge0\}$ with Hurst index $H\in(0,1)$ is the centred Gaussian process with $B^H(0)=0$ and covariance $\E[B^H(t)B^H(s)]=\frac12(t^{2H}+s^{2H}-|t-s|^{2H})$ \cite{MandelbrotVanNess1968}. Superpositions of independent fBms with different indices describe phenomena with several sources of dependence; the mixed fBm of Cheridito \cite{Cheridito2001} and its generalizations \cite{AlajmiMliki2021} are the best known examples, and Maleki Almani and Sottinen \cite{MaleAlmaniSottinen2023} study the finite superpositions, which they call multi-mixed fBm. Models of this kind are used, for example, for network traffic \cite{VidacsVirtamo1999}, for stock indices \cite{DufitinemaPynnonenSottinen2022,SunChenLu2022} and for degradation data \cite{BorosEtAl2024}. In this paper we consider
\begin{equation}\label{eq:model}
Z(t)=\theta t+\alpha\sum_{r=1}^{m}B^{H_r}(t),\qquad t\ge0,
\end{equation}
with independent fBms $B^{H_1},\dots,B^{H_m}$, a drift $\theta\in\R$, a common scale $\alpha>0$, and observations $Z(t_1),\dots,Z(t_N)$ at $t_k=kh$.

When the Hurst vector $\bH=(H_1,\dots,H_m)$ is known, inference on $(\theta,\alpha^2)$ is classical. The increments of $Z$ form a stationary Gaussian sequence with mean $\theta h$ and covariance matrix known up to the factor $\alpha^2$, so the maximum likelihood estimator of $\theta$ is the generalized least squares (GLS) estimator and the Gaussian linear model gives exact Student and chi-square intervals \cite{LehmannCasella1998,LehmannRomano2005}. Drift estimation for Gaussian processes with known covariance, including sums of fBms and processes with non-stationary increments, is developed in \cite{MishuraRalchenkoShklyar2017,MishuraRalchenkoShklyar2018a,MishuraRalchenkoShklyar2018b}. For discrete observations and a known Hurst index, explicit maximum likelihood estimators of the drift and the scale and their asymptotic normality are obtained for fBm \cite{HuNualartXiaoZhang2011}, sub-fractional Brownian motion \cite{KuangLiu2015} and $n$th-order fBm \cite{ChaouchElMaroufyElOmari2023}, with Berry--Esseen bounds in the first two cases. The two scales of a mixed fBm are estimated in \cite{XiaoZhangZhang2011}, and tests on the drift of an fBm are studied in \cite{Stiburek2017}. For continuous observations, the maximum likelihood estimator of the drift is constructed for fBm in \cite{NorrosValkeilaVirtamo1999}, for sums of two fBms in \cite{Mishura2016,MishuraVoronov2015}, where it involves the solution of an integral equation that is computed numerically in \cite{MishuraRalchenkoZhelezniak2024}, and for the mixed fractional Vasicek model in \cite{CaiHuangSunXiao2022}. Random-walk approximations of the noise are used in \cite{BertinTorresTudor2011a,BertinTorresTudor2011b}, and the monograph \cite{KubiliusMishuraRalchenko2017} surveys the field. We record the oracle theory in Proposition~\ref{prop:oracle} and use it only as a benchmark; Remark~\ref{rem:exact} shows that the distributional results just cited for known indices also follow from it.

In applications the Hurst indices are not known. They are estimated, by maximum likelihood \cite{HarbaDouziElHajji2012,VidacsVirtamo1999} or by other methods, including deep learning \cite{BorosEtAl2024}, and then plugged into the model. The question we address is what remains of the oracle inference when $\bH$ is replaced by a wrong value or by an estimate. Joint estimation of all parameters has been studied for two-component models. Dufitinema, Pynn\"onen and Sottinen \cite{DufitinemaPynnonenSottinen2022} maximize the profile likelihood of $\theta t+\sigma W_t+\kappa B^H_t$ and apply it to Nordic stock indices, but their asymptotic results for the drift and the scale assume that the Hurst index and the mixing ratio $\kappa/\sigma$ are known. Kukush et al.\ \cite{KukushEtAl2022} and Ralchenko and Yakovliev \cite{RalchenkoYakovliev2023} treat the same model with power-variation and ergodic-type estimators, Ralchenko and Yakovliev \cite{RalchenkoYakovliev2024} treat $\kappa B^{H_1}_t+\sigma B^{H_2}_t$, with joint asymptotic normality for $0<H_1<H_2<3/4$, and Ketrat and Tanadkithirun \cite{KetratTanadkithirun2026} prove strong consistency for $\mu t+\sigma W_t+\alpha B^{H_1}_t+\beta B^{H_2}_t$. For an $n$th-order mixed fBm with polynomial drift, El Omari \cite{ElOmari2023,ElOmari2025} first estimates some of the parameters by power variations or by the ergodic theorem, the Hurst index among them in \cite{ElOmari2025}, and then the drift by maximum likelihood, which is a plug-in procedure. These works construct consistent and asymptotically normal estimators, and \cite{RalchenkoYakovliev2024} reports the simulated coverage of the resulting asymptotic intervals; the effect of estimating the Hurst indices on the exact intervals for the drift and the scale has not been studied. Unlike the oracle theory, this question depends on the fractional structure of the covariance, through the behaviour of the spectral density at the origin.

Our contributions are the following.
\begin{enumerate}
\item Misspecification (Theorem~\ref{thm:misspec}). If a wrong Hurst vector $\tbH$ is used, the actual variance of the working drift estimator divided by the variance claimed by the working model tends to infinity when $\tHmax<\Hmax$, with a lower bound of order $N^{2(\Hmax-\tHmax)}$, and to zero when $\tHmax>\Hmax$. Hence the coverage of the Student interval tends to $0$ or to $1$, and the chi-square interval for $\alpha^2$ fails as well. Numerically the variance ratio grows exactly like $N^{2(\Hmax-\tHmax)}$ in both directions (Section~\ref{sec:sim}).
\item Estimation of the Hurst vector (Proposition~\ref{prop:ident}, Theorem~\ref{thm:mle}, Proposition~\ref{prop:info}). The model is identifiable, and a two-step maximum likelihood estimator, with the drift first estimated by $Z(t_N)/t_N$, is consistent, asymptotically normal and efficient for all indices in $(1/2,1)$; this follows from the theory of Lieberman, Rosemarin and Rousseau \cite{LiebermanRosemarinRousseau2012}, which extends Dahlhaus \cite{Dahlhaus1989,Dahlhaus2006}. The information matrix degenerates like the square of the smallest gap between indices, so clustered indices are weakly identified.
\item Plug-in inference for the drift (Theorem~\ref{thm:plugdrift}). The plug-in GLS estimator is asymptotically equivalent to the oracle one and the plug-in Student interval is asymptotically valid. The estimated standard error, however, carries a relative error of order $N^{-1/2}\log N$, and a Satterthwaite-type correction of the degrees of freedom restores the level in finite samples (Remark~\ref{rem:satt}).
\item Plug-in inference for the scale (Theorem~\ref{thm:plugscale}). The plug-in chi-square interval for $\alpha^2$ is asymptotically invalid. Its coverage tends to $2\Phi(z_{1-\beta/2}/\sqrt{\VIF})-1$ with an explicit variance inflation factor $\VIF\ge1$; for one component $\VIF=1+(2\log h+\bar\ell_H)^2/v_H$, which equals $10.1$ for $H=0.7$ and monthly data ($h=1/12$) and $35.0$ for daily data ($h=1/252$). We give an asymptotically valid replacement.
\end{enumerate}

Section~\ref{sec:model} introduces the model and the oracle inference. Section~\ref{sec:misspec} treats a misspecified Hurst vector, Section~\ref{sec:mle} the estimation of the Hurst vector, and Section~\ref{sec:plugin} plug-in inference. Section~\ref{sec:sim} reports a simulation study and Section~\ref{sec:concl} concludes.

\section{Model, spectral facts and oracle inference}\label{sec:model}

\subsection{The observations}
Throughout, $m\ge1$ is known and the Hurst vector satisfies the following assumption.

\begin{assumption}\label{ass:H}
$1/2<H_1<H_2<\dots<H_m<1$. We write $\Hmax=H_m$.
\end{assumption}

Let $X_k=Z(t_k)-Z(t_{k-1})$, $k=1,\dots,N$, with $t_0=0$, and $X=(X_1,\dots,X_N)^\top$. Since the vector $(Z(t_1),\dots,Z(t_N))$ and $X$ determine each other, all inference is based on $X$. We have
\begin{equation}\label{eq:incr}
X=\theta h\,\one+\alpha Y,\qquad Y\sim N\bigl(0,C_N(\bH)\bigr),
\end{equation}
where $\one=(1,\dots,1)^\top$. For an integrable even function $g$ on $[-\pi,\pi]$ let $T_N(g)$ be the Toeplitz matrix with entries $\int_{-\pi}^{\pi}e^{i(j-k)\lambda}g(\lambda)\,d\lambda$, $1\le j,k\le N$. Let $f_H$ be the spectral density of fractional Gaussian noise with unit variance, so that
$$
\gamma_H(k)=\tfrac12\bigl(|k+1|^{2H}+|k-1|^{2H}-2|k|^{2H}\bigr)=\int_{-\pi}^{\pi}e^{ik\lambda}f_H(\lambda)\,d\lambda .
$$
By independence and self-similarity, $C_N(\bH)=T_N(f(\cdot;\bH))$ with
\begin{equation}\label{eq:f}
f(\lambda;\bH)=\sum_{r=1}^{m}h^{2H_r}f_{H_r}(\lambda),
\end{equation}
so $X$ is a stationary Gaussian sequence with mean $\theta h$ and spectral density $\alpha^2f(\cdot;\bH)$.

For a symmetric positive definite matrix $C$ we write
$$
V(C)=\bigl(\one^\top C^{-1}\one\bigr)^{-1},\qquad w(C)=V(C)\,C^{-1}\one .
$$
The vector $w(C)$ gives the best linear unbiased estimator (BLUE) $w(C)^\top X$ of the mean of a sequence with covariance $C$, and $V(C)$ is its variance. We abbreviate $V_N(\bH)=V(C_N(\bH))$ and $w_N(\bH)=w(C_N(\bH))$. For a vector $w$ we write $W(\lambda)=\sum_kw_ke^{ik\lambda}$, so that $w^\top T_N(g)w=\int_{-\pi}^{\pi}|W(\lambda)|^2g(\lambda)\,d\lambda$. In particular
\begin{equation}\label{eq:eigen}
2\pi\inf g\le \lambda_{\min}\bigl(T_N(g)\bigr)\le\lambda_{\max}\bigl(T_N(g)\bigr)\le 2\pi\sup g .
\end{equation}

\subsection{Spectral facts}
The following lemma collects the properties of $f_H$ that we use.

\begin{lemma}\label{lem:spec}
Let $c_H=\Gamma(2H+1)\sin(\pi H)/(2\pi)$.
\begin{enumerate}
\item[(a)] $f_H(\lambda)=2c_H(1-\cos\lambda)\sum_{j\in\Z}|\lambda+2\pi j|^{-2H-1}$, and $f_H(\lambda)=c_H|\lambda|^{1-2H}+O\bigl(|\lambda|^{\min(3-2H,2)}\bigr)$ as $\lambda\to0$.
\item[(b)] If $H\in(1/2,1)$, then $f_H$ is continuous and positive on $[-\pi,\pi]\setminus\{0\}$, $f_H(\lambda)\to\infty$ as $\lambda\to0$, and $\inf f_H>0$.
\item[(c)] For every compact $\mathcal K\subset(0,1)$ and $k\in\{1,2,3\}$ there is a constant $K$ such that $|\partial_H^kf_H(\lambda)|\le K(1+|\log|\lambda||)^kf_H(\lambda)$ for all $H\in\mathcal K$ and $0<|\lambda|\le\pi$.
\end{enumerate}
\end{lemma}

\begin{proof}
(a) The series is the spectral density of fractional Gaussian noise with unit variance; see \cite[Proposition~2.1]{Beran1994}. We write
$$
f_H(\lambda)=c_H|\lambda|^{1-2H}+R_H(\lambda)
$$
and estimate the remainder $R_H$. The term $j=0$ of the series equals
$$
2c_H(1-\cos\lambda)|\lambda|^{-2H-1}=c_H|\lambda|^{1-2H}g(\lambda),\qquad g(\lambda)=\frac{2(1-\cos\lambda)}{\lambda^2},
$$
where $g$ is smooth, even and free of $H$, and $g(\lambda)=1+O(\lambda^2)$. For $j\ne0$ and $|\lambda|\le\pi$ we have $|\lambda+2\pi j|\ge\pi$, so the sum over $j\ne0$ is a smooth function of $(\lambda,H)$ on $[-\pi,\pi]\times(0,1)$. Since this sum is multiplied by $2c_H(1-\cos\lambda)=O(\lambda^2)$, we obtain
$$
R_H(\lambda)=c_H|\lambda|^{1-2H}\bigl(g(\lambda)-1\bigr)+2c_H(1-\cos\lambda)\sum_{j\ne0}|\lambda+2\pi j|^{-2H-1}=O\bigl(|\lambda|^{\min(3-2H,2)}\bigr),
$$
which is the expansion in (a). Moreover, $\partial_H|\lambda|^{1-2H}=-2\log|\lambda|\,|\lambda|^{1-2H}$, and the other factors are smooth in $H$. Hence $\partial_HR_H(\lambda)$ satisfies the same bound with an extra factor $1+|\log|\lambda||$. Both bounds are uniform for $H$ in compact subsets of $(0,1)$.

(b) Since $c_H>0$, $1-\cos\lambda>0$ for $0<|\lambda|\le\pi$, and every term of the series is positive, $f_H$ is positive on $[-\pi,\pi]\setminus\{0\}$. Next, for $|\lambda|\le\pi$,
$$
\sum_{j\ne0}|\lambda+2\pi j|^{-2H-1}\le2\sum_{j\ge1}\bigl(\pi(2j-1)\bigr)^{-2H-1}<\infty,
$$
so the sum over $j\ne0$ converges uniformly, and therefore $f_H$ is continuous on $[-\pi,\pi]\setminus\{0\}$. Since $H>1/2$, the exponent $1-2H$ is negative, so $f_H(\lambda)\to\infty$ as $\lambda\to0$ by part (a). Choose $\delta>0$ such that $f_H\ge1$ on $\{0<|\lambda|<\delta\}$. On the compact set $\{\delta\le|\lambda|\le\pi\}$ the function $f_H$ is continuous and positive, so it has a positive minimum there. Consequently $\inf f_H>0$.

(c) Every compact subset of $(0,1)$ lies in an interval $[H_-,H_+]\subset(0,1)$, so we may take $\mathcal K=[H_-,H_+]$. We write
$$
f_H(\lambda)=2(1-\cos\lambda)\,c_H\,S_H(\lambda),\qquad S_H(\lambda)=\sum_{j\in\Z}|\lambda+2\pi j|^{-2H-1}.
$$
The factor $2(1-\cos\lambda)$ does not depend on $H$. Also, $c_H$ is smooth and positive on $\mathcal K$, so $|\partial_H^ic_H|\le Kc_H$ for $i\le3$. It therefore suffices to prove that
\begin{equation}\label{eq:SH}
|\partial_H^kS_H(\lambda)|\le K(1+|\log|\lambda||)^kS_H(\lambda),\qquad k\le3 .
\end{equation}
Indeed, the Leibniz rule and \eqref{eq:SH} then give
$$
|\partial_H^k(c_HS_H)|\le\sum_{i=0}^k\binom ki|\partial_H^ic_H|\,|\partial_H^{k-i}S_H|\le K(1+|\log|\lambda||)^k\,c_HS_H,
$$
and (c) follows after multiplication by $2(1-\cos\lambda)$.

To prove \eqref{eq:SH}, we differentiate term by term:
$$
\partial_H^k|\lambda+2\pi j|^{-2H-1}=\bigl(-2\log|\lambda+2\pi j|\bigr)^k|\lambda+2\pi j|^{-2H-1}.
$$
This is justified, because apart from the term $j=0$ the differentiated series converges uniformly on $[-\pi,\pi]\times\mathcal K$. For $j=0$ the logarithmic factor is at most $2^k|\log|\lambda||^k$. For $j\ne0$ and $|\lambda|\le\pi$ we have $\pi\le|\lambda+2\pi j|\le3\pi|j|$, hence $0<\log|\lambda+2\pi j|\le\log(3\pi|j|)$; also $|\lambda+2\pi j|\ge\pi(2|j|-1)>1$. Therefore
$$
\sum_{j\ne0}\log^k(3\pi|j|)\,|\lambda+2\pi j|^{-2H-1}\le\sum_{j\ne0}\log^k(3\pi|j|)\bigl(\pi(2|j|-1)\bigr)^{-2H_--1}=:K_1<\infty .
$$
On the other hand, the term $j=1$ alone shows that $\sum_{j\ne0}|\lambda+2\pi j|^{-2H-1}\ge(3\pi)^{-3}$. Hence the terms with $j\ne0$ contribute at most $2^kK_1(3\pi)^3\sum_{j\ne0}|\lambda+2\pi j|^{-2H-1}$ to $|\partial_H^kS_H|$, and \eqref{eq:SH} follows.
\end{proof}

Under Assumption~\ref{ass:H}, Lemma~\ref{lem:spec} gives the corresponding facts for \eqref{eq:f}: $f(\cdot;\bH)$ is continuous and positive away from $0$, $\inf f(\cdot;\bH)>0$,
\begin{equation}\label{eq:f0}
f(\lambda;\bH)\sim h^{2\Hmax}c_{\Hmax}|\lambda|^{1-2\Hmax}\qquad(\lambda\to0),
\end{equation}
and $|\partial_{H_r}f(\lambda;\bH)|\le K(1+|\log|\lambda||)f(\lambda;\bH)$, because $\partial_{H_r}f=h^{2H_r}(2\log h\,f_{H_r}+\partial_Hf_{H_r})$ and $h^{2H_r}f_{H_r}\le f$. These bounds are uniform on compact sets of Hurst vectors satisfying Assumption~\ref{ass:H}.

The variance of the BLUE depends only on the spectral density near the origin; this is the theorem of Adenstedt \cite{Adenstedt1974}.

\begin{lemma}\label{lem:adenstedt}
Let $\bH$ satisfy Assumption~\ref{ass:H} and let $g$ be bounded, bounded away from zero, and continuous at $0$. Then
$$
V\bigl(T_N(f(\cdot;\bH)\,g)\bigr)\sim g(0)\,\kappa(\bH)\,N^{2\Hmax-2},\qquad
\kappa(\bH)=\frac{2\pi h^{2\Hmax}c_{\Hmax}\,\Gamma(2-2\Hmax)}{B(\tfrac32-\Hmax,\tfrac32-\Hmax)} .
$$
In particular $V_N(\bH)\sim\kappa(\bH)N^{2\Hmax-2}$.
\end{lemma}

\begin{proof}
We use the theorem of Adenstedt \cite{Adenstedt1974} in the following form. Let $f_a(\lambda)=(2\pi)^{-1}|1-e^{i\lambda}|^{2a}$ with $a>-1/2$, and let $F=f_ag_1$, where $g_1$ is bounded, bounded away from zero and continuous at $0$. Then
$$
V\bigl(T_N(F)\bigr)\sim\frac{\Gamma(2a+1)\,g_1(0)}{B(a+1,a+1)}\,N^{-2a-1}.
$$
We apply this with $a=\frac12-\Hmax\in(-\frac12,0)$ and $g_1=f(\cdot;\bH)\,g/f_a$, so that $F=f(\cdot;\bH)\,g$.

By Lemma~\ref{lem:spec}, the ratio $f(\cdot;\bH)/f_a$ is continuous and positive on $[-\pi,\pi]\setminus\{0\}$. Moreover $|1-e^{i\lambda}|\sim|\lambda|$ as $\lambda\to0$, so \eqref{eq:f0} gives
$$
\frac{f(\lambda;\bH)}{f_a(\lambda)}\to2\pi h^{2\Hmax}c_{\Hmax}\qquad(\lambda\to0).
$$
Hence $f(\cdot;\bH)/f_a$ is bounded and bounded away from zero. Therefore $g_1$ is bounded, bounded away from zero and continuous at $0$, with $g_1(0)=2\pi h^{2\Hmax}c_{\Hmax}\,g(0)$. Finally, $-2a-1=2\Hmax-2$, $2a+1=2-2\Hmax$ and $a+1=\frac32-\Hmax$, which gives the claim.
\end{proof}

For $m=1$, $H=0.7$ and $h=1$ the constant is $\kappa=0.98654$, while $N^{0.6}V_N$ equals $0.98667$, $0.98657$ and $0.98655$ for $N=250$, $1000$ and $4000$. For several components the convergence is slower, because a component with index $H_r<\Hmax$ contributes a relative correction of order $N^{-2(\Hmax-H_r)}$.

\subsection{Oracle inference}
For a Hurst vector $\bH$ satisfying Assumption~\ref{ass:H} put $C=C_N(\bH)$ and
\begin{equation}\label{eq:gls}
\hat\theta(\bH)=\frac{\one^\top C^{-1}X}{h\,\one^\top C^{-1}\one},\qquad
\hat\alpha^2(\bH)=\frac1N\bigl(X-h\hat\theta\one\bigr)^\top C^{-1}\bigl(X-h\hat\theta\one\bigr),\qquad
d_N(\bH)=\frac{h^2}{V_N(\bH)},
\end{equation}
and $\tilde\alpha^2(\bH)=N\hat\alpha^2(\bH)/(N-1)$. For $\beta\in(0,1)$ define the Student and chi-square intervals
$$
I_\theta(\bH)=\Bigl[\hat\theta\mp t_{N-1,1-\beta/2}\,\tilde\alpha/\sqrt{d_N}\Bigr],\qquad
J_\alpha(\bH)=\Bigl[\frac{N\hat\alpha^2}{\chi^2_{N-1,1-\beta/2}},\frac{N\hat\alpha^2}{\chi^2_{N-1,\beta/2}}\Bigr],
$$
where $t_{k,p}$ and $\chi^2_{k,p}$ denote quantiles.

\begin{proposition}\label{prop:oracle}
If $\bH$ is the true Hurst vector, then $\hat\theta(\bH)\sim N(\theta,\alpha^2/d_N(\bH))$, $N\hat\alpha^2(\bH)/\alpha^2\sim\chi^2_{N-1}$, the two are independent, and $\Prob(\theta\in I_\theta(\bH))=\Prob(\alpha^2\in J_\alpha(\bH))=1-\beta$ for every $N\ge2$. Moreover $\Var(\hat\theta(\bH))=\alpha^2V_N(\bH)/h^2\sim\alpha^2\kappa(\bH)h^{-2}N^{2\Hmax-2}$.
\end{proposition}

\begin{proof}
Put $X^*=C^{-1/2}X$ and $u=h\,C^{-1/2}\one$. By \eqref{eq:incr},
$$
X^*=\theta u+\alpha\varepsilon,\qquad \varepsilon=C^{-1/2}Y\sim N(0,I_N),
$$
which is the Gaussian linear model with one known regressor $u$ and error variance $\alpha^2$. In this model the least squares estimator of $\theta$ is $u^\top X^*/u^\top u=\hat\theta(\bH)$, and the residual sum of squares is
$$
|X^*-\hat\theta u|^2=\bigl(X-h\hat\theta\one\bigr)^\top C^{-1}\bigl(X-h\hat\theta\one\bigr)=N\hat\alpha^2(\bH).
$$
Hence, by the classical theory of the Gaussian linear model \cite[Ch.~3]{LehmannCasella1998}, \cite[Ch.~5]{LehmannRomano2005}, $\hat\theta(\bH)\sim N(\theta,\alpha^2/u^\top u)$, $N\hat\alpha^2(\bH)/\alpha^2\sim\chi^2_{N-1}$, and the two are independent; for drift estimation in this setting see also \cite{HuNualartXiaoZhang2011,MishuraRalchenkoShklyar2017}. Since $u^\top u=h^2\one^\top C^{-1}\one=d_N(\bH)$, it follows that
$$
\frac{\sqrt{d_N(\bH)}\,(\hat\theta-\theta)}{\tilde\alpha}\sim t_{N-1},
$$
and therefore $\Prob(\theta\in I_\theta(\bH))=1-\beta$. Also, $\alpha^2\in J_\alpha(\bH)$ if and only if $\chi^2_{N-1,\beta/2}\le N\hat\alpha^2/\alpha^2\le\chi^2_{N-1,1-\beta/2}$, and this event has probability $1-\beta$. Finally, $\Var(\hat\theta(\bH))=\alpha^2/d_N(\bH)=\alpha^2V_N(\bH)/h^2$, and Lemma~\ref{lem:adenstedt} gives its asymptotic behaviour.
\end{proof}

\begin{remark}\label{rem:exact}
The proof of Proposition~\ref{prop:oracle} uses only that the covariance matrix of the observations is known up to the factor $\alpha^2$. The same exact laws therefore hold for any Gaussian noise whose covariance is known up to a scale factor, for instance a sub-fractional Brownian motion \cite{KuangLiu2015} or an $n$th-order fBm \cite{ChaouchElMaroufyElOmari2023} with known indices, with $h\one$ replaced by the corresponding regressor. In particular, the asymptotic normality of the drift and scale estimators proved in \cite{ChaouchElMaroufyElOmari2023,HuNualartXiaoZhang2011,KuangLiu2015}, and the Berry--Esseen bounds of \cite{HuNualartXiaoZhang2011,KuangLiu2015}, can be read off from these laws. Indeed, the standardized drift estimator is exactly standard normal, and $N\hat\alpha^2/\alpha^2$ has the law of a sum of $N-1$ independent $\chi^2_1$ variables, so the classical Berry--Esseen theorem gives the rate $N^{-1/2}$ for its normal approximation. What these exact results do not cover is the effect of an unknown Hurst vector, which is the subject of the rest of the paper.
\end{remark}

\begin{remark}\label{rem:endpoint}
The endpoint estimator $\bar\theta=Z(t_N)/t_N=(Nh)^{-1}\sum_kX_k$ satisfies
$$
\Var(\bar\theta)=\alpha^2\sum_{r=1}^{m}(Nh)^{2H_r-2}
$$
exactly, since $\Var(B^H(Nh))=(Nh)^{2H}$. It has the optimal rate and does not depend on $\bH$; we use it as a pilot estimator in Section~\ref{sec:mle}. The bound $\Var(\hat\theta(\bH))\le\Var(\bar\theta)$ is the one used in \cite{MishuraRalchenkoShklyar2018a}.
\end{remark}

\section{Misspecified Hurst vector}\label{sec:misspec}

Suppose that a working Hurst vector $\tbH=(\widetilde H_1,\dots,\widetilde H_{\tilde m})$ satisfying Assumption~\ref{ass:H} is used in place of the true vector $\bH$, and write $C=C_N(\bH)$, $\widetilde C=C_N(\tbH)$, $f=f(\cdot;\bH)$, $\tilde f=f(\cdot;\tbH)$, $\tilde w=w(\widetilde C)$. The working estimators are $\tilde\theta=\hat\theta(\tbH)$ and $\tilde\alpha^2=\tilde\alpha^2(\tbH)$, the working intervals are $I_\theta(\tbH)$ and $J_\alpha(\tbH)$, and $\widetilde T=\sqrt{d_N(\tbH)}(\tilde\theta-\theta)/\tilde\alpha$ is the working Student statistic. Define
$$
A_N=\frac{\one^\top\widetilde C^{-1}C\widetilde C^{-1}\one}{\one^\top\widetilde C^{-1}\one}=\frac{\tilde w^\top C\tilde w}{V(\widetilde C)},\qquad
B_N=\frac{\tr(\widetilde PC)}{N-1},\qquad
\widetilde P=\widetilde C^{-1}-\frac{\widetilde C^{-1}\one\one^\top\widetilde C^{-1}}{\one^\top\widetilde C^{-1}\one}.
$$
$A_N$ is the actual variance of $\tilde w^\top Y$ divided by the variance claimed by the working model, and $\alpha^2B_N$ is the mean of $\tilde\alpha^2$.

\begin{lemma}\label{lem:AB}
Let $Q=Y^\top\widetilde PY$.
\begin{enumerate}
\item[(a)] $\widetilde T=\xi\sqrt{A_N}\big/\sqrt{Q/(N-1)}$ with $\xi\sim N(0,1)$, $\tilde\alpha^2=\alpha^2Q/(N-1)$ and $\E Q=(N-1)B_N$.
\item[(b)] $A_N\ge V(C)/V(\widetilde C)$.
\item[(c)] $(N-1)B_N\le N\sum_rh^{2H_r}/(2\pi\inf\tilde f)$ and $(N-1)B_N\ge 2\pi N\inf f/\sum_sh^{2\widetilde H_s}-A_N$.
\item[(d)] $\Var(Q)\le 2\tr(C^2)/(2\pi\inf\tilde f)^2$ and $\tr(C^2)=o(N^2)$.
\end{enumerate}
\end{lemma}

\begin{proof}
(a) Since $\tilde\theta=\tilde w^\top X/h$ and $\one^\top\tilde w=1$, \eqref{eq:incr} gives $\tilde\theta-\theta=(\alpha/h)\,\tilde w^\top Y$. Also $d_N(\tbH)=h^2/V(\widetilde C)$. Hence
$$
\frac{\sqrt{d_N(\tbH)}\,(\tilde\theta-\theta)}{\alpha}=\frac{\tilde w^\top Y}{\sqrt{V(\widetilde C)}},
$$
which is a centred Gaussian variable with variance $\tilde w^\top C\tilde w/V(\widetilde C)=A_N$. Next, the residual is $X-h\tilde\theta\one=\alpha\bigl(Y-(\tilde w^\top Y)\one\bigr)$. Since $\tilde w^\top Y=\one^\top\widetilde C^{-1}Y/\one^\top\widetilde C^{-1}\one$, a direct expansion gives
$$
\bigl(Y-(\tilde w^\top Y)\one\bigr)^\top\widetilde C^{-1}\bigl(Y-(\tilde w^\top Y)\one\bigr)=Y^\top\widetilde C^{-1}Y-\frac{(\one^\top\widetilde C^{-1}Y)^2}{\one^\top\widetilde C^{-1}\one}=Y^\top\widetilde PY .
$$
Therefore $(N-1)\tilde\alpha^2=N\hat\alpha^2(\tbH)=\alpha^2Q$, and $\E Q=\tr(\widetilde PC)=(N-1)B_N$. Finally, with $\xi=\tilde w^\top Y/(\tilde w^\top C\tilde w)^{1/2}\sim N(0,1)$, we obtain $\widetilde T=\xi\sqrt{A_N}/\sqrt{Q/(N-1)}$.

(b) By the Gauss--Markov theorem, $V(C)=\min\{v^\top Cv:\one^\top v=1\}$. Since $\one^\top\tilde w=1$, it follows that $\tilde w^\top C\tilde w\ge V(C)$, that is, $A_N\ge V(C)/V(\widetilde C)$.

(c) Since $\widetilde P\preceq\widetilde C^{-1}$ and $\lambda_{\max}(\widetilde C^{-1})\le1/(2\pi\inf\tilde f)$ by \eqref{eq:eigen}, we have
$$
(N-1)B_N=\tr(\widetilde PC)\le\lambda_{\max}(\widetilde C^{-1})\tr C\le\frac{\tr C}{2\pi\inf\tilde f}.
$$
Moreover $\tr C=N\sum_rh^{2H_r}$, because $\gamma_H(0)=1$; this gives the upper bound. For the lower bound, the definitions of $\widetilde P$ and $A_N$ give $\tr(\widetilde PC)=\tr(\widetilde C^{-1}C)-A_N$. Also $\lambda_{\min}(C)\ge2\pi\inf f$ by \eqref{eq:eigen}, and $\tr(\widetilde C^{-1})\ge N^2/\tr\widetilde C$ by the Cauchy--Schwarz inequality. Hence
$$
\tr(\widetilde C^{-1}C)\ge\lambda_{\min}(C)\tr(\widetilde C^{-1})\ge\frac{2\pi N^2\inf f}{\tr\widetilde C}=\frac{2\pi N\inf f}{\sum_sh^{2\widetilde H_s}},
$$
and the lower bound follows.

(d) Since $Y\sim N(0,C)$, we have $\Var(Q)=2\tr\bigl((\widetilde PC)^2\bigr)$. Moreover
$$
0\preceq C^{1/2}\widetilde PC^{1/2}\preceq C^{1/2}\widetilde C^{-1}C^{1/2}\preceq\lambda_{\max}(\widetilde C^{-1})\,C,
$$
and $0\preceq A\preceq B$ implies $\tr(A^2)\le\tr(B^2)$. Hence
$$
\Var(Q)\le2\lambda_{\max}(\widetilde C^{-1})^2\tr(C^2)\le\frac{2\tr(C^2)}{(2\pi\inf\tilde f)^2}.
$$
It remains to bound $\tr(C^2)$. The entries of $C$ are $c(j-k)$, where $c(k)=\sum_rh^{2H_r}\gamma_{H_r}(k)=O\bigl((1+|k|)^{2\Hmax-2}\bigr)$. Therefore
$$
\tr(C^2)=\sum_{|k|<N}(N-|k|)\,c(k)^2\le N\sum_{|k|<N}c(k)^2=O\bigl(N\max\{\log N,N^{4\Hmax-3}\}\bigr),
$$
and this is $o(N^2)$ because $4\Hmax-3<1$.
\end{proof}

The next lemma says that the variance of the BLUE comes from arbitrarily low frequencies.

\begin{lemma}\label{lem:conc}
Let $\tbH$ satisfy Assumption~\ref{ass:H}, and let $\varphi\ge0$ be bounded and measurable on $[-\pi,\pi]$ and vanish on a neighbourhood of $0$. Then $\int_{-\pi}^{\pi}|\widetilde W(\lambda)|^2\tilde f(\lambda)\varphi(\lambda)\,d\lambda=o\bigl(V(\widetilde C)\bigr)$, where $\widetilde W$ is the transfer function of $\tilde w$.
\end{lemma}

\begin{proof}
Put $D=\int_{-\pi}^{\pi}|\widetilde W|^2\tilde f\varphi\,d\lambda$. If $\varphi=0$ there is nothing to prove, so let $0<s<1/\sup\varphi$ and $g=1-s\varphi$. Then $1-s\sup\varphi\le g\le1$, and $g=1$ near $0$. Since $\one^\top\tilde w=1$, the vector $\tilde w$ is admissible in the minimization that defines $V(T_N(\tilde fg))$, and therefore
$$
V\bigl(T_N(\tilde fg)\bigr)\le\tilde w^\top T_N(\tilde fg)\,\tilde w=\int_{-\pi}^{\pi}|\widetilde W|^2\tilde f(1-s\varphi)\,d\lambda=V(\widetilde C)-sD .
$$
On the other hand, Lemma~\ref{lem:adenstedt}, applied to $\tbH$ and $g$, gives $V(T_N(\tilde fg))\sim g(0)V(\widetilde C)=V(\widetilde C)$. Hence
$$
sD\le V(\widetilde C)-V\bigl(T_N(\tilde fg)\bigr)=o\bigl(V(\widetilde C)\bigr),
$$
which proves the lemma.
\end{proof}

\begin{theorem}\label{thm:misspec}
Let $\bH$ and $\tbH$ satisfy Assumption~\ref{ass:H} and let $\beta\in(0,1)$.
\begin{enumerate}
\item[(a)] If $\tHmax<\Hmax$, then $A_N\ge c\,N^{2(\Hmax-\tHmax)}$ for some $c>0$ and all large $N$, and $\Prob(\theta\in I_\theta(\tbH))\to0$.
\item[(b)] If $\tHmax>\Hmax$, then $A_N\to0$ and $\Prob(\theta\in I_\theta(\tbH))\to1$.
\item[(c)] If $\tHmax=\Hmax$, then $1\le\liminf A_N\le\limsup A_N<\infty$.
\item[(d)] In all cases $B_N$ is bounded and $Q/(N-1)-B_N\to0$ in probability. If $\liminf_N|B_N-1|>0$, then $\Prob(\alpha^2\in J_\alpha(\tbH))\to0$.
\end{enumerate}
\end{theorem}

\begin{proof}
Write $q_N=t_{N-1,1-\beta/2}$. Then $q_N\to z_{1-\beta/2}$, and $\theta\in I_\theta(\tbH)$ if and only if $|\widetilde T|\le q_N$.

(a) By Lemma~\ref{lem:AB}(b) and Lemma~\ref{lem:adenstedt},
$$
A_N\ge\frac{V_N(\bH)}{V_N(\tbH)}\sim\frac{\kappa(\bH)}{\kappa(\tbH)}\,N^{2(\Hmax-\tHmax)},
$$
which gives the lower bound for $A_N$; in particular $A_N\to\infty$. Now fix $M>0$. On the event $\{Q/(N-1)\le M\}$, Lemma~\ref{lem:AB}(a) shows that $|\widetilde T|\le q_N$ implies $|\xi|\le q_N\sqrt{M/A_N}$. Therefore, by Markov's inequality and since the density of $\xi$ is at most $(2\pi)^{-1/2}$,
$$
\Prob(|\widetilde T|\le q_N)\le\Prob\Bigl(\frac{Q}{N-1}>M\Bigr)+\Prob\Bigl(|\xi|\le q_N\sqrt{\frac M{A_N}}\Bigr)\le\frac{\sup_NB_N}{M}+q_N\sqrt{\frac{2M}{\pi A_N}} .
$$
Here $\sup_NB_N<\infty$ by Lemma~\ref{lem:AB}(c). Letting first $N\to\infty$ and then $M\to\infty$, we obtain $\Prob(\theta\in I_\theta(\tbH))\to0$.

(b) Let $\rho=f/\tilde f$. By Lemma~\ref{lem:spec}, $\rho$ is continuous on $[-\pi,\pi]\setminus\{0\}$, and by \eqref{eq:f0}, applied to $\bH$ and to $\tbH$,
$$
\rho(\lambda)\sim\frac{h^{2\Hmax}c_{\Hmax}}{h^{2\tHmax}c_{\tHmax}}\,|\lambda|^{2(\tHmax-\Hmax)}\to0\qquad(\lambda\to0).
$$
Hence $\rho$ is bounded. Let $\varepsilon>0$, choose $\eta>0$ such that $\rho\le\varepsilon$ on $\{|\lambda|\le\eta\}$, and put $\varphi=\mathbf 1_{\{|\lambda|>\eta\}}$. Since $\int|\widetilde W|^2\tilde f\,d\lambda=V(\widetilde C)$, Lemma~\ref{lem:conc} gives
$$
\tilde w^\top C\tilde w=\int_{-\pi}^{\pi}|\widetilde W|^2\rho\tilde f\,d\lambda\le\varepsilon V(\widetilde C)+\sup\rho\int_{-\pi}^{\pi}|\widetilde W|^2\tilde f\varphi\,d\lambda=\varepsilon V(\widetilde C)+o\bigl(V(\widetilde C)\bigr).
$$
Therefore $\limsup A_N\le\varepsilon$, and since $\varepsilon$ is arbitrary, $A_N\to0$.

Next, Lemma~\ref{lem:AB}(c) and $A_N\to0$ give $b>0$ such that $B_N\ge b$ for all large $N$. Also $\E[Q/(N-1)]=B_N$ and $\Var(Q/(N-1))\to0$ by Lemma~\ref{lem:AB}(a) and (d), so Chebyshev's inequality gives $\Prob(Q/(N-1)\ge b/2)\to1$. On this event, Lemma~\ref{lem:AB}(a) yields
$$
|\widetilde T|\le|\xi|\sqrt{2A_N/b},
$$
and the right side tends to $0$ in probability. Hence $\Prob(|\widetilde T|\le q_N)\to1$.

(c) Now $\tHmax=\Hmax$, so $\rho$ is bounded and, by \eqref{eq:f0}, $\rho(\lambda)\to1$ as $\lambda\to0$. Since $f\le(\sup\rho)\tilde f$, we have $C\preceq(\sup\rho)\,\widetilde C$, and hence
$$
A_N=\frac{\tilde w^\top C\tilde w}{V(\widetilde C)}\le\sup\rho\cdot\frac{\tilde w^\top\widetilde C\tilde w}{V(\widetilde C)}=\sup\rho .
$$
For the lower bound, recall that $\kappa(\bH)$ depends on $\bH$ only through $\Hmax$, so $\kappa(\bH)=\kappa(\tbH)$. Therefore Lemma~\ref{lem:AB}(b) and Lemma~\ref{lem:adenstedt} give
$$
\liminf_{N\to\infty}A_N\ge\lim_{N\to\infty}\frac{V_N(\bH)}{V_N(\tbH)}=1 .
$$

(d) By Lemma~\ref{lem:AB}(c), $B_N$ is bounded. Moreover $\E[Q/(N-1)]=B_N$ and $\Var(Q/(N-1))\to0$ by Lemma~\ref{lem:AB}(a) and (d); hence $Q/(N-1)-B_N\to0$ in $L^2$, and therefore in probability.

For the scale interval, Lemma~\ref{lem:AB}(a) gives $N\hat\alpha^2(\tbH)=\alpha^2Q$, and hence
$$
\alpha^2\in J_\alpha(\tbH)\iff\chi^2_{N-1,\beta/2}\le Q\le\chi^2_{N-1,1-\beta/2}.
$$
Now suppose that $\liminf_N|B_N-1|>0$, and choose $\varepsilon>0$ such that $|B_N-1|\ge\varepsilon$ for all large $N$. Since $\chi^2_{N-1,p}/(N-1)\to1$ for every $p\in(0,1)$, on the event above we have $|Q/(N-1)-1|\le\varepsilon/2$ for large $N$, and therefore
$$
\Bigl|\frac Q{N-1}-B_N\Bigr|\ge|B_N-1|-\Bigl|\frac Q{N-1}-1\Bigr|\ge\frac\varepsilon2 .
$$
Since $Q/(N-1)-B_N\to0$ in probability, the probability of this event tends to zero. Hence $\Prob(\alpha^2\in J_\alpha(\tbH))\to0$.
\end{proof}

\begin{remark}\label{rem:misspec}
(i) The proof of (a) gives the lower bound $A_N\gtrsim N^{2(\Hmax-\tHmax)}$. Numerically $A_N$ is of exact order $N^{2(\Hmax-\tHmax)}$ in both directions; the fitted exponents in Table~\ref{tab:misspec} differ from $2(\Hmax-\tHmax)$ by at most $0.0004$. The working estimator $\tilde\theta$ itself keeps the optimal rate in our examples; what fails is the variance that the working model attributes to it. Adenstedt \cite{Adenstedt1974} also studied the efficiency of estimators whose weights are computed with a wrong exponent; Theorem~\ref{thm:misspec} concerns instead the studentized statistic, whose behaviour is governed by $A_N$.
(ii) By Szeg\H{o}-type theorems one expects $B_N\to(2\pi)^{-1}\int_{-\pi}^{\pi}f/\tilde f\,d\lambda$; we do not need this. In all our examples $B_N$ is essentially constant in $N$ and different from $1$, so the chi-square interval loses its coverage for every fixed error in the Hurst vector.
\end{remark}

\begin{remark}\label{rem:tests}
Theorem~\ref{thm:misspec} also describes tests on the drift. The Student test of the hypothesis $\theta=\theta_0$ at level $\beta$ rejects it when $\theta_0\notin I_\theta(\tbH)$. Hence, when the hypothesis is true, the rejection probability tends to $1$ if $\tHmax<\Hmax$ and to $0$ if $\tHmax>\Hmax$. The same holds when $\alpha$ is known and the test uses the Gaussian statistic $\sqrt{d_N(\tbH)}(\tilde\theta-\theta_0)/\alpha$, because under the hypothesis this statistic is centred Gaussian with variance $A_N$ by Lemma~\ref{lem:AB}(a); tests of this kind for the drift of an fBm with a known Hurst index are studied in \cite{Stiburek2017}. Thus an underestimated largest index makes a nominal level-$\beta$ test reject a true hypothesis with probability tending to one, while an overestimated one makes the test increasingly conservative.
\end{remark}

\section{Estimating the Hurst vector}\label{sec:mle}

\begin{assumption}\label{ass:Psi}
The parameter $\psi=(\alpha^2,\bH)$ ranges over a compact set $\Psi=[a_1,a_2]\times\mathcal H$, where $0<a_1<a_2<\infty$ and $\mathcal H$ is a compact set of vectors satisfying Assumption~\ref{ass:H}. The true value $\psi_0=(\alpha_0^2,\bH_0)$ is an interior point of $\Psi$.
\end{assumption}

\begin{proposition}\label{prop:ident}
Let $(\alpha^2,\bH)$ and $(\alpha'^2,\bH')$ be two parameters with $\bH$, $\bH'$ satisfying Assumption~\ref{ass:H}. If $\alpha^2f(\lambda;\bH)=\alpha'^2f(\lambda;\bH')$ for almost every $\lambda$, then $\alpha^2=\alpha'^2$ and $\bH=\bH'$.
\end{proposition}

\begin{proof}
By Lemma~\ref{lem:spec}(b), both functions are continuous on $(0,\pi]$; hence they coincide there. We compare their behaviour as $\lambda\downarrow0$. By \eqref{eq:f0},
$$
\alpha^2f(\lambda;\bH)\sim\alpha^2h^{2H_m}c_{H_m}\lambda^{1-2H_m},\qquad
\alpha'^2f(\lambda;\bH')\sim\alpha'^2h^{2H'_{m'}}c_{H'_{m'}}\lambda^{1-2H'_{m'}} .
$$
Since the two functions are equal, the exponents agree, so $H_m=H'_{m'}$; then the constants agree as well, and therefore $\alpha^2=\alpha'^2$. Removing the common top component, we obtain
$$
\sum_{r<m}h^{2H_r}f_{H_r}(\lambda)=\sum_{s<m'}h^{2H'_s}f_{H'_s}(\lambda),\qquad0<\lambda\le\pi .
$$
By Lemma~\ref{lem:spec}(a), $f_H(\lambda)\sim c_H\lambda^{1-2H}$ as $\lambda\downarrow0$ for every $H\in(1/2,1)$. Hence the same argument shows that the largest remaining indices coincide, and we repeat this step. Since every $f_H$ is positive, a nonempty sum cannot vanish; therefore both sums are exhausted at the same step. Consequently $m=m'$ and $\bH=\bH'$.
\end{proof}

\begin{definition}\label{def:est}
The two-step estimator is defined as follows.
\begin{enumerate}
\item Compute the endpoint estimator $\bar\theta=Z(t_N)/t_N$ of Remark~\ref{rem:endpoint}.
\item Let $\hat\psi=(\hat\alpha^2_{(2)},\hbH)$ maximize over $\Psi$ the Gaussian log-likelihood with the mean fixed at $h\bar\theta$,
$$
\ell_N(\psi)=-\frac12\log\det\bigl(\alpha^2C_N(\bH)\bigr)-\frac1{2\alpha^2}\bigl(X-h\bar\theta\one\bigr)^\top C_N(\bH)^{-1}\bigl(X-h\bar\theta\one\bigr).
$$
For fixed $\bH$ the maximizing $\alpha^2$ is $s_N^2(\bH)=(X-h\bar\theta\one)^\top C_N(\bH)^{-1}(X-h\bar\theta\one)/N$, so $\hbH$ maximizes the concentrated likelihood $-\frac N2\log s_N^2(\bH)-\frac12\log\det C_N(\bH)$, and $\hat\alpha^2_{(2)}=s_N^2(\hbH)$.
\item Compute the plug-in GLS quantities $\hat\theta=\hat\theta(\hbH)$, $\hat\alpha^2(\hbH)$, $\tilde\alpha^2(\hbH)$ and $d_N(\hbH)$ of \eqref{eq:gls}.
\end{enumerate}
\end{definition}

\begin{remark}\label{rem:profile}
For a single fBm, replacing $\bar\theta$ in step~2 by the GLS estimator $\hat\theta(\bH)$ gives the profile likelihood used by Vid\'acs and Virtamo \cite{VidacsVirtamo1999} for fractional Brownian traffic, and Dufitinema, Pynn\"onen and Sottinen \cite{DufitinemaPynnonenSottinen2022} maximize the analogous profile likelihood for a mixed fBm. We use the pilot estimator $\bar\theta$ instead, because it does not depend on $\bH$; this allows us to apply the results of \cite{LiebermanRosemarinRousseau2012} on Gaussian maximum likelihood with a plug-in mean. Maximum likelihood estimation of the Hurst index of an fBm from irregularly sampled data is discussed in \cite{HarbaDouziElHajji2012}.
\end{remark}

The information matrix of the spectral parameter is
\begin{equation}\label{eq:info}
\mathcal I(\psi)=\frac1{4\pi}\int_{-\pi}^{\pi}\nabla_\psi\log\bigl(\alpha^2f(\lambda;\bH)\bigr)\,\nabla_\psi\log\bigl(\alpha^2f(\lambda;\bH)\bigr)^\top d\lambda,
\end{equation}
with components $\partial_{\alpha^2}\log(\alpha^2f)=\alpha^{-2}$ and $\ell_r(\lambda;\bH):=\partial_{H_r}\log f(\lambda;\bH)$.

\begin{theorem}\label{thm:mle}
Under Assumption~\ref{ass:Psi}, $\hat\psi\to\psi_0$ in probability and $\sqrt N(\hat\psi-\psi_0)\dto N\bigl(0,\mathcal I(\psi_0)^{-1}\bigr)$, where $\mathcal I(\psi_0)$ is positive definite. The estimator $\hat\psi$ is asymptotically efficient.
\end{theorem}

\begin{proof}
We apply Theorem~1 of Lieberman, Rosemarin and Rousseau \cite{LiebermanRosemarinRousseau2012} to the stationary Gaussian sequence $X$ with mean $\mu=\theta h$, spectral density $\alpha^2f(\cdot;\bH)$ and plug-in mean $\hat\mu=h\bar\theta$. Their Gaussian maximum likelihood estimator with plug-in mean is exactly $\hat\psi$, so it suffices to check their Assumptions 0, 1, 2, 4 and 5.

\emph{Assumption 0.} It requires identifiability, which is Proposition~\ref{prop:ident}, and compactness of $\Psi$ with $\psi_0$ an interior point, which is Assumption~\ref{ass:Psi}. It also requires
$$
\alpha^2f(\lambda;\bH)\sim|\lambda|^{-a(\psi)}L_\psi(\lambda)\qquad(\lambda\to0),
$$
with $L_\psi$ slowly varying and $a(\psi)<1$. By \eqref{eq:f0}, this holds with $a(\psi)=2\Hmax-1\in(0,1)$ and $L_\psi(\lambda)\to\alpha^2h^{2\Hmax}c_{\Hmax}$. Below we write $a=a(\psi)$.

\emph{Assumption 1.} It requires that $f$ be continuous on $[-\pi,\pi]\setminus\{0\}$ and that, for every $\delta>0$,
$$
f=O\bigl(|\lambda|^{-a-\delta}\bigr),\qquad f^{-1}=O\bigl(|\lambda|^{a-\delta}\bigr),\qquad\partial_\lambda f=O\bigl(|\lambda|^{-a-1-\delta}\bigr).
$$
The continuity and the first two bounds follow from \eqref{eq:f0} and Lemma~\ref{lem:spec}(b). For the third bound, we differentiate the series of Lemma~\ref{lem:spec}(a) term by term. The term $j=0$ equals $2c_H(1-\cos\lambda)|\lambda|^{-2H-1}$, and its derivative is $O(|\lambda|^{-2H})$; since $H\le\Hmax$ and $-2\Hmax=-a-1$, this is $O(|\lambda|^{-a-1})$. The remaining terms have bounded derivatives.

\emph{Assumption 2.} It requires $\partial_\psi^kf=O(|\lambda|^{-a-\delta})$ for $k\le3$ and every $\delta>0$. Derivatives in $\alpha^2$ are trivial. A derivative in $\bH$ of order $k$ is a finite sum of terms
$$
h^{2H_r}(\log h)^i\,\partial_H^jf_{H_r},\qquad i+j=k,
$$
and by Lemma~\ref{lem:spec}(c) each of them is $O\bigl((1+|\log|\lambda||)^kf\bigr)$. Since $(1+|\log|\lambda||)^k=O(|\lambda|^{-\delta})$, they are $O(|\lambda|^{-a-\delta})$.

\emph{Assumption 4.} It requires $a(\psi)$ to be continuous and the constants in the previous bounds to be uniform over $\Psi$. Both hold, because $\Hmax$ is a coordinate of $\bH$, $\Psi$ is compact, and the bounds of Lemma~\ref{lem:spec} are uniform on compact sets.

\emph{Assumption 5.} It requires $\hat\mu-\mu=o_p\bigl(N^{(a(\psi_0)-1)/2+\delta}\bigr)$ for every $\delta>0$. By Remark~\ref{rem:endpoint}, $\hat\mu-\mu=h(\bar\theta-\theta)$ is centred Gaussian with standard deviation $O(N^{\Hmax-1})$, and $(a(\psi_0)-1)/2=\Hmax-1$. Hence this assumption holds as well.

Therefore Theorem~1 of \cite{LiebermanRosemarinRousseau2012} applies, and it gives the consistency, the asymptotic normality and the efficiency of $\hat\psi$. Finally, $\mathcal I(\psi_0)$ is positive definite by Proposition~\ref{prop:info}.
\end{proof}

\begin{proposition}\label{prop:info}
Under Assumption~\ref{ass:H}, $\mathcal I(\psi)$ is positive definite. If moreover $\bH$ ranges over a compact set of vectors in $(1/2,1)^m$ with increasing coordinates, there is a constant $K$ such that $\lambda_{\min}(\mathcal I(\psi))\le K\,(H_{r+1}-H_r)^2$ for every $r<m$.
\end{proposition}

\begin{proof}
Suppose that $c_0\alpha^{-2}+\sum_rc_r\ell_r=0$ almost everywhere. Since $\ell_r=h^{2H_r}\bigl(2\log h\,f_{H_r}+\partial_Hf_{H_r}\bigr)/f$, multiplying by $f(\lambda;\bH)$ gives
\begin{equation}\label{eq:indep}
\sum_{r=1}^mh^{2H_r}\Bigl[\bigl(c_0\alpha^{-2}+2c_r\log h\bigr)f_{H_r}(\lambda)+c_r\,\partial_Hf_{H_r}(\lambda)\Bigr]=0,\qquad0<\lambda\le\pi .
\end{equation}
By Lemma~\ref{lem:spec}(a) and its proof, for every $H\in(1/2,1)$,
$$
f_H(\lambda)=c_H\lambda^{1-2H}+R_H(\lambda),\qquad
\partial_Hf_H(\lambda)=c_H\lambda^{1-2H}\bigl(\partial_H\log c_H-2\log\lambda\bigr)+\partial_HR_H(\lambda),
$$
where $R_H$ and $\partial_HR_H$ are $O(\lambda^{\min(3-2H,2)}|\log\lambda|)$ as $\lambda\downarrow0$. Since $1-2H_r<0$, these remainders are $o(\lambda^{1-2H_r})$ for every $r$. Hence, as $\lambda\downarrow0$, the most singular term in \eqref{eq:indep} is
$$
-2c_m\,h^{2H_m}c_{H_m}\,\lambda^{1-2H_m}\log\lambda .
$$
Dividing \eqref{eq:indep} by $\lambda^{1-2H_m}\log\lambda$ and letting $\lambda\downarrow0$, we obtain $c_m=0$. Then the coefficient of $\lambda^{1-2H_m}$ is $h^{2H_m}c_{H_m}c_0\alpha^{-2}$, and the same argument gives $c_0=0$. Now \eqref{eq:indep} contains only the components $r<m$, and the terms $\lambda^{1-2H_r}\log\lambda$ give $c_{m-1}=\dots=c_1=0$ in turn. Therefore the scores $\alpha^{-2},\ell_1,\dots,\ell_m$ are linearly independent in $L^2(d\lambda)$, and $\mathcal I(\psi)$ is positive definite.

For the second statement, fix $r<m$ and let $u=e_{r+1}-e_r$, where $e_r$ is the unit vector of $H_r$. Then
$$
u^\top\mathcal I(\psi)u=\frac1{4\pi}\int_{-\pi}^{\pi}(\ell_{r+1}-\ell_r)^2\,d\lambda .
$$
Put $G(H)=\partial_H\bigl(h^{2H}f_H\bigr)$, so that $\ell_r=G(H_r)/f$. By the mean value theorem, for each $\lambda$ there is $H^*\in(H_r,H_{r+1})$ with
$$
\ell_{r+1}-\ell_r=(H_{r+1}-H_r)\,\frac{G'(H^*)}{f}.
$$
Since $G'(H)=h^{2H}\bigl[(2\log h)^2f_H+4\log h\,\partial_Hf_H+\partial_H^2f_H\bigr]$, Lemma~\ref{lem:spec}(c) gives
$$
|\ell_{r+1}-\ell_r|\le K(H_{r+1}-H_r)\,(1+|\log|\lambda||)^2\,\frac{h^{2H^*}f_{H^*}}{f}.
$$
Moreover $f\ge h^{2H_{r+1}}f_{H_{r+1}}$, and the ratio $f_{H^*}/f_{H_{r+1}}$ is bounded, because by Lemma~\ref{lem:spec}(a) it behaves like $|\lambda|^{2(H_{r+1}-H^*)}$ at the origin. Hence $h^{2H^*}f_{H^*}/f\le K$, and therefore
$$
u^\top\mathcal I(\psi)u\le K(H_{r+1}-H_r)^2\int_{-\pi}^{\pi}(1+|\log|\lambda||)^4\,d\lambda\le K(H_{r+1}-H_r)^2 .
$$
Finally, $\lambda_{\min}(\mathcal I(\psi))\le u^\top\mathcal I(\psi)u/|u|^2$, which proves the claim.
\end{proof}

\begin{remark}\label{rem:weak}
By the Cauchy--Schwarz inequality, $u^\top\mathcal I^{-1}u\ge|u|^4/u^\top\mathcal Iu$, so the asymptotic variance of $\sqrt N(\hat H_{r+1}-\hat H_r)$ is at least $c\,(H_{r+1}-H_r)^{-2}$. Resolving two indices at distance $g$ therefore needs $N$ large compared with $g^{-4}$. Table~\ref{tab:weakid} shows that in practice even well-separated triples are hard to estimate, especially for small $h$: the factor $h^{2H_r}$ in \eqref{eq:f} favours the components with small indices, so the components with large indices dominate only at very low frequencies. This agrees with the simulations of Ralchenko and Yakovliev \cite{RalchenkoYakovliev2024} for $\kappa B^{H_1}+\sigma B^{H_2}$, in which the coverage of the asymptotic intervals is closer to the nominal level the larger $H_2-H_1$ is.
\end{remark}

\section{Plug-in inference}\label{sec:plugin}

\subsection{The drift}
The next lemma controls how the BLUE depends on the Hurst vector.

\begin{lemma}\label{lem:blue}
Let $\mathcal K$ be a compact convex set of vectors satisfying Assumption~\ref{ass:H}. There is a constant $K$ such that for all $N\ge2$, $r\le m$ and $\bH,\bH'\in\mathcal K$:
\begin{enumerate}
\item[(a)] $|\partial_{H_r}\log V_N(\bH)|\le K\log N$;
\item[(b)] if $|\bH-\bH'|\log N\le1$, then $\bigl|w_N(\bH')^\top C_N(\bH)w_N(\bH')/V_N(\bH)-1\bigr|\le K|\bH-\bH'|\log N$.
\end{enumerate}
\end{lemma}

\begin{proof}
Fix $\bH\in\mathcal K$, let $w=w_N(\bH)$ with transfer function $W$, and write $f=f(\cdot;\bH)$. Since $\int|W|^2f\,d\lambda=w^\top C_N(\bH)w=V_N(\bH)$, the formula
$$
\mu(d\lambda)=\frac{|W(\lambda)|^2f(\lambda)}{V_N(\bH)}\,d\lambda
$$
defines a probability measure on $[-\pi,\pi]$. We first show that $\mu$ puts little mass near the origin. By the Cauchy--Schwarz inequality, $\sup|W|^2\le(\sum_k|w_k|)^2\le N|w|^2$. Also, by \eqref{eq:eigen},
$$
|w|^2\le\frac{w^\top C_N(\bH)w}{2\pi\inf f}=\frac{V_N(\bH)}{2\pi\inf f}.
$$
Hence $\sup|W|^2\le NV_N(\bH)/(2\pi\inf f)$, and therefore, for $0<\eta\le\pi$,
\begin{equation}\label{eq:tail}
\mu(|\lambda|\le\eta)\le\frac{N}{2\pi\inf f}\int_{|\lambda|\le\eta}f(\lambda)\,d\lambda\le KN\eta^{b},\qquad b=2-2\Hmax,
\end{equation}
where the last step uses \eqref{eq:f0}. Here $K$ and $b\ge b_0>0$ are uniform on $\mathcal K$.

Let $L(\lambda)=\max\{0,\log(1/|\lambda|)\}$. For $s\ge0$ we have $\{L>s\}=\{|\lambda|<e^{-s}\}$, so \eqref{eq:tail} gives $\mu(L>s)\le\min\{1,KNe^{-bs}\}$. Integrating this tail bound, we obtain
\begin{equation}\label{eq:mgf}
\E_\mu L\le\frac{\log(KN)+1}{b},\qquad \E_\mu e^{aL}\le(KN)^{a/b}\Bigl(1+\frac a{b-a}\Bigr)\quad(0\le a<b).
\end{equation}
Indeed, put $s_0=\log(KN)/b$. Since $\mu(L>s)\le1$ for all $s$ and $\mu(L>s)\le KNe^{-bs}$, we have
$$
\E_\mu L=\int_0^\infty\mu(L>s)\,ds\le s_0+\int_{s_0}^\infty KNe^{-bs}\,ds=s_0+\frac1b,
$$
and similarly
$$
\E_\mu e^{aL}=1+\int_0^\infty ae^{as}\mu(L>s)\,ds\le e^{as_0}+\int_{s_0}^\infty aKNe^{-(b-a)s}\,ds=(KN)^{a/b}\Bigl(1+\frac a{b-a}\Bigr).
$$

(a) Since $V_N=(\one^\top C_N^{-1}\one)^{-1}$ and $w=V_NC_N^{-1}\one$, differentiation gives
$$
\partial_{H_r}V_N=V_N^2\,\one^\top C_N^{-1}(\partial_{H_r}C_N)C_N^{-1}\one=w^\top(\partial_{H_r}C_N)w=\int_{-\pi}^{\pi}|W|^2\,\partial_{H_r}f\,d\lambda .
$$
By the bound after \eqref{eq:f0}, $|\partial_{H_r}f|\le K(1+|\log|\lambda||)f$. Moreover, $|\log|\lambda||\le\log\pi+L(\lambda)$ for $|\lambda|\le\pi$. Hence, by \eqref{eq:mgf},
$$
|\partial_{H_r}\log V_N|\le K\,\E_\mu\bigl(1+|\log|\lambda||\bigr)\le K\bigl(1+\log\pi+\E_\mu L\bigr)\le K\log N .
$$

(b) Let $w'=w_N(\bH')$, and let $\mu'$ be defined as $\mu$ with $\bH'$ in place of $\bH$. Since $w'^\top C_N(\bH)w'=\int|W'|^2f(\cdot;\bH)\,d\lambda$, we have
\begin{equation}\label{eq:ratio}
\frac{w'^\top C_N(\bH)w'}{V_N(\bH)}=\frac{V_N(\bH')}{V_N(\bH)}\;\E_{\mu'}\Bigl[\frac{f(\cdot;\bH)}{f(\cdot;\bH')}\Bigr].
\end{equation}
We bound the two factors. By Lemma~\ref{lem:spec}(c), applied along the segment from $\bH'$ to $\bH$, which lies in $\mathcal K$ by convexity,
$$
\Bigl|\log\frac{f(\lambda;\bH)}{f(\lambda;\bH')}\Bigr|\le a\bigl(1+\log\pi+L(\lambda)\bigr),\qquad a=K|\bH-\bH'| .
$$
Since $|\bH-\bH'|\le1/\log N$, we have $a\le b_0/2$ for $N\ge N_0$. Then \eqref{eq:mgf} bounds the expectation in \eqref{eq:ratio} from above, and Jensen's inequality bounds it from below; together,
$$
\Bigl|\log\E_{\mu'}\Bigl[\frac{f(\cdot;\bH)}{f(\cdot;\bH')}\Bigr]\Bigr|\le K|\bH-\bH'|\log N .
$$
Also, part (a) and the mean value theorem give $|\log(V_N(\bH')/V_N(\bH))|\le K|\bH-\bH'|\log N$. Therefore the logarithm of \eqref{eq:ratio} is at most $K|\bH-\bH'|\log N\le K$ in absolute value. Since $|e^x-1|\le e^K|x|$ for $|x|\le K$, this proves (b) for $N\ge N_0$. Finally, for each of the finitely many $N<N_0$, the ratio \eqref{eq:ratio} is a smooth function of $(\bH,\bH')$ that equals $1$ on the diagonal, so the bound holds after enlarging $K$.
\end{proof}

\begin{theorem}\label{thm:plugdrift}
Under Assumption~\ref{ass:Psi}, let $\hat\theta=\hat\theta(\hbH)$, $\hat d=d_N(\hbH)$ and $\hat\sigma^2=\tilde\alpha^2(\hbH)$ be as in Definition~\ref{def:est}. Then:
\begin{enumerate}
\item[(a)] $\sqrt{d_N(\bH_0)}\,\bigl(\hat\theta-\hat\theta(\bH_0)\bigr)/\alpha_0\to0$ in probability; in particular $\hat\theta$ is asymptotically equivalent to the oracle estimator and asymptotically efficient;
\item[(b)] $\log\bigl(\hat d/d_N(\bH_0)\bigr)=O_p(N^{-1/2}\log N)$ and $\hat\sigma^2\to\alpha_0^2$ in probability;
\item[(c)] $\widehat T=\sqrt{\hat d}\,(\hat\theta-\theta)/\hat\sigma\dto N(0,1)$; consequently $\Prob(\theta\in I_\theta(\hbH))\to1-\beta$.
\end{enumerate}
\end{theorem}

\begin{proof}
Let $\mathcal K\subset\mathcal H$ be a compact convex neighbourhood of $\bH_0$. By Theorem~\ref{thm:mle}, $\Prob(\hbH\in\mathcal K)\to1$ and $\hbH-\bH_0=O_p(N^{-1/2})$.

(a) For $\bH'\in\mathcal K$ put
$$
U(\bH')=\frac{w_N(\bH')^\top Y}{\sqrt{V_N(\bH_0)}} .
$$
Since $\hat\theta(\bH')-\theta=(\alpha_0/h)\,w_N(\bH')^\top Y$ and $d_N(\bH_0)=h^2/V_N(\bH_0)$, we have
$$
\frac{\sqrt{d_N(\bH_0)}\,\bigl(\hat\theta(\bH')-\theta\bigr)}{\alpha_0}=U(\bH'),
$$
and in particular $U(\bH_0)\sim N(0,1)$. Thus (a) states that $U(\hbH)-U(\bH_0)\to0$ in probability. We prove this in three steps.

\emph{Step 1: variance of the increments.} Let $C=C_N(\bH_0)$ and $w_0=w_N(\bH_0)$. Since $Cw_0=V_N(\bH_0)\one$, we have $(v-w_0)^\top Cw_0=0$ whenever $\one^\top v=1$. Hence, by Lemma~\ref{lem:blue}(b),
$$
\E\bigl(U(\bH')-U(\bH_0)\bigr)^2=\frac{w_N(\bH')^\top Cw_N(\bH')}{V_N(\bH_0)}-1\le K|\bH'-\bH_0|\log N
$$
whenever $|\bH'-\bH_0|\log N\le1$.

\emph{Step 2: a finite grid.} Let $r_N=N^{-1/2}\log N$, and let $\mathcal B_N$ be the ball of radius $r_N$ around $\bH_0$. Since $r_N\log N\to0$, Step~1 applies on $\mathcal B_N$ for large $N$. Let $G_N\subset\mathcal B_N$ be a finite set such that every point of $\mathcal B_N$ lies within distance $N^{-4}$ of $G_N$, with $|G_N|\le KN^{4m}$. Each $U(g)-U(\bH_0)$ is centred Gaussian with variance at most $Kr_N\log N=KN^{-1/2}(\log N)^2$. Therefore the Gaussian tail bound and the union bound give, for every $\varepsilon>0$,
$$
\Prob\Bigl(\max_{g\in G_N}|U(g)-U(\bH_0)|>\varepsilon\Bigr)\le2KN^{4m}\exp\Bigl(-\frac{\varepsilon^2N^{1/2}}{2K(\log N)^2}\Bigr)\to0 .
$$

\emph{Step 3: from the grid to the ball.} We use a crude Lipschitz bound. For $\bH'\in\mathcal K$ let $w=w_N(\bH')$ and $C'=C_N(\bH')$. Differentiating $w=V_NC'^{-1}\one$ gives
$$
\partial_{H_r}w=(\partial_{H_r}\log V_N)\,w-C'^{-1}(\partial_{H_r}C')\,w .
$$
We bound each factor. First, taking $v=N^{-1}\one$ in the minimization that defines $V_N(\bH')$, we get
$$
V_N(\bH')\le\frac{\one^\top C'\one}{N^2}=\sum_rh^{2H_r}N^{2H_r-2}\le K,
$$
and hence $|w|^2\le V_N(\bH')/(2\pi\inf f)\le K$, as in the proof of Lemma~\ref{lem:blue}. Next, $\|C'^{-1}\|\le K$ by \eqref{eq:eigen}. Finally, since $|\partial_H\gamma_H(k)|\le K(1+k)^{2H-2}\log(2+k)$,
$$
\|\partial_{H_r}C'\|\le2\sum_{k<N}|\partial_{H_r}c(k)|\le KN\log N .
$$
Together with Lemma~\ref{lem:blue}(a), this gives $|\partial_{H_r}w|\le KN\log N$. Moreover, $|Y|=O_p(N^{1/2})$, and $V_N(\bH_0)^{-1/2}\le KN^{1-\Hmax}\le KN^{1/2}$ by Lemma~\ref{lem:adenstedt}. Since every point of $\mathcal B_N$ lies within distance $N^{-4}$ of $G_N$, it follows that
$$
\sup_{\bH'\in\mathcal B_N}\,\min_{g\in G_N}|U(\bH')-U(g)|\le\frac{KN^{-4}\cdot N\log N\cdot|Y|}{\sqrt{V_N(\bH_0)}}=O_p(N^{-2}\log N).
$$

Combining Steps 2 and 3, we get $\sup_{\bH'\in\mathcal B_N}|U(\bH')-U(\bH_0)|\to0$ in probability. Since $\hbH-\bH_0=O_p(N^{-1/2})$, also $\Prob(\hbH\in\mathcal B_N)\to1$, and hence $U(\hbH)-U(\bH_0)\to0$ in probability. This proves the asymptotic equivalence. Moreover, $\hat\theta(\bH_0)$ is unbiased with variance $\alpha_0^2/d_N(\bH_0)$, which is the Cramér--Rao bound for $\theta$, and in the Gaussian likelihood the information for $\theta$ is orthogonal to that for $(\alpha^2,\bH)$. Therefore $\hat\theta$ is asymptotically efficient.

(b) Since $\log d_N=2\log h-\log V_N$, Lemma~\ref{lem:blue}(a) and the mean value theorem give, on the event $\{\hbH\in\mathcal K\}$,
$$
\Bigl|\log\frac{\hat d}{d_N(\bH_0)}\Bigr|\le K\log N\,|\hbH-\bH_0|=O_p(N^{-1/2}\log N).
$$
For the second claim, let $\hat\alpha^2_{(2)}$ be the scale estimator of step~2 of Definition~\ref{def:est}; it converges to $\alpha_0^2$ by Theorem~\ref{thm:mle}. Since $\hat\theta$ is the GLS estimator for $C_N(\hbH)$, the residual $X-h\hat\theta\one$ is orthogonal to $\one$ in the inner product defined by $C_N(\hbH)^{-1}$. Hence the Pythagorean identity of least squares gives
$$
N\hat\alpha^2_{(2)}-(N-1)\hat\sigma^2=(\bar\theta-\hat\theta)^2\,\hat d .
$$
We show that the right side is $O_p(1)$. By Remark~\ref{rem:endpoint}, $\bar\theta-\theta=O_p(N^{\Hmax-1})$. By part (a) and Proposition~\ref{prop:oracle}, also $\hat\theta-\theta=O_p(N^{\Hmax-1})$. Hence $(\bar\theta-\hat\theta)^2=O_p(N^{2\Hmax-2})$. On the other hand, $\hat d=O_p(N^{2-2\Hmax})$ by the first claim and Lemma~\ref{lem:adenstedt}. Therefore the right side is $O_p(1)$, and
$$
\hat\sigma^2=\frac{N\hat\alpha^2_{(2)}+O_p(1)}{N-1}\to\alpha_0^2\quad\text{in probability.}
$$

(c) By the definitions of $\widehat T$ and $U$,
$$
\widehat T=\frac{\sqrt{\hat d}\,(\hat\theta-\theta)}{\hat\sigma}=U(\hbH)\,\sqrt{\frac{\hat d}{d_N(\bH_0)}}\;\frac{\alpha_0}{\hat\sigma}.
$$
By (a), $U(\hbH)=U(\bH_0)+o_p(1)$ with $U(\bH_0)\sim N(0,1)$, and by (b) the last two factors tend to $1$ in probability. Hence $\widehat T\dto N(0,1)$ by Slutsky's lemma. Finally, $\theta\in I_\theta(\hbH)$ if and only if $|\widehat T|\le t_{N-1,1-\beta/2}$, and $t_{N-1,1-\beta/2}\to z_{1-\beta/2}$; therefore $\Prob(\theta\in I_\theta(\hbH))\to1-\beta$.
\end{proof}

\begin{remark}\label{rem:satt}
Theorem~\ref{thm:plugdrift}(b) shows that replacing $\bH_0$ by $\hbH$ perturbs the estimated variance $\hat v=\hat\sigma^2/\hat d$ of $\hat\theta$ by a relative amount of order $N^{-1/2}\log N$, larger than the fluctuation of order $N^{-1/2}$ for which the Student law accounts. A Satterthwaite-type correction \cite{Satterthwaite1946} matches the variance of $\log\hat v$ with that of $\log(\chi^2_\nu/\nu)$, which is approximately $2/\nu$. With $\hat g=\nabla_\psi\log\bigl(\alpha^2V_N(\bH)\bigr)$ evaluated at $\hat\psi$, the delta method gives
$$
\hat\nu=\frac2{\hat s^2},\qquad\hat s^2=\hat g^\top\bigl(N\mathcal I(\hat\psi)\bigr)^{-1}\hat g,
$$
and the corrected interval replaces $t_{N-1,1-\beta/2}$ by $t_{\hat\nu,1-\beta/2}$. By Lemma~\ref{lem:blue}(a), $|\hat g|=O_p(\log N)$, so $\hat\nu$ is of order at least $N/(\log N)^2$ and the correction vanishes asymptotically. In practice $\nabla_{\bH}\log V_N$ is obtained by numerical differentiation of $\log V_N$, which costs one Cholesky factorization per coordinate and direction. Section~\ref{sec:sim} shows that the correction restores the nominal level for one component; for several components with close indices $\hat\nu$ is small, which signals that the standard error of $\hat\theta$ is itself poorly determined.
\end{remark}

\subsection{The scale}

For $\bH$ satisfying Assumption~\ref{ass:H} define, with $\ell_r=\partial_{H_r}\log f(\cdot;\bH)$,
$$
\bar m_r(\bH)=\frac1{2\pi}\int_{-\pi}^{\pi}\ell_r\,d\lambda,\qquad
S_{rs}(\bH)=\frac1{2\pi}\int_{-\pi}^{\pi}(\ell_r-\bar m_r)(\ell_s-\bar m_s)\,d\lambda,\qquad
\VIF(\bH)=1+\bar m^\top S^{-1}\bar m .
$$
The matrix $S$ is positive definite by the proof of Proposition~\ref{prop:info}, and $\bar m=\nabla_{\bH}\bar\kappa$ with $\bar\kappa(\bH)=(2\pi)^{-1}\int_{-\pi}^{\pi}\log f(\lambda;\bH)\,d\lambda$, by differentiation under the integral, which is justified by Lemma~\ref{lem:spec}(c).

\begin{theorem}\label{thm:plugscale}
Under Assumption~\ref{ass:Psi}, $\sqrt N\bigl(\hat\alpha^2(\hbH)-\alpha_0^2\bigr)\dto N\bigl(0,2\alpha_0^4\VIF(\bH_0)\bigr)$. Consequently
$$
\Prob\bigl(\alpha_0^2\in J_\alpha(\hbH)\bigr)\to2\Phi\Bigl(\frac{z_{1-\beta/2}}{\sqrt{\VIF(\bH_0)}}\Bigr)-1,
$$
which is smaller than $1-\beta$ whenever $\nabla\bar\kappa(\bH_0)\ne0$, while the interval
$$
J^{\mathrm c}_\alpha=\Bigl[\hat\alpha^2(\hbH)\,e^{-z_{1-\beta/2}\sqrt{2\VIF(\hbH)/N}},\ \hat\alpha^2(\hbH)\,e^{z_{1-\beta/2}\sqrt{2\VIF(\hbH)/N}}\Bigr]
$$
has asymptotic coverage $1-\beta$. For one component,
\begin{equation}\label{eq:vif1}
\VIF(H)=1+\frac{(2\log h+\bar\ell_H)^2}{v_H},
\end{equation}
where $\bar\ell_H=(2\pi)^{-1}\int_{-\pi}^{\pi}\partial_H\log f_H\,d\lambda$ and $v_H=(2\pi)^{-1}\int_{-\pi}^{\pi}(\partial_H\log f_H-\bar\ell_H)^2d\lambda$.
\end{theorem}

\begin{proof}
By the proof of Theorem~\ref{thm:plugdrift}(b), $N\hat\alpha^2_{(2)}-N\hat\alpha^2(\hbH)=O_p(1)$, that is, $\hat\alpha^2(\hbH)$ and $\hat\alpha^2_{(2)}$ differ by $O_p(N^{-1})$. Hence, by Theorem~\ref{thm:mle}, $\sqrt N(\hat\alpha^2(\hbH)-\alpha_0^2)$ is asymptotically centred normal with variance $[\mathcal I(\psi_0)^{-1}]_{\alpha^2\alpha^2}$. We compute this entry.

Let $\E_u$ denote the average over $\lambda$ uniform on $[-\pi,\pi]$, and let $s=(\alpha^{-2},\ell_1,\dots,\ell_m)^\top$ be the vector of scores. Then \eqref{eq:info} reads $\mathcal I=\frac12\E_u[ss^\top]$. Since $\E_u\ell=\bar m$ and $\E_u[\ell\ell^\top]=S+\bar m\bar m^\top$, the blocks of $\mathcal I$ are
$$
\mathcal I_{\alpha^2\alpha^2}=\tfrac12\alpha^{-4},\qquad\mathcal I_{\alpha^2\bH}=\tfrac12\alpha^{-2}\bar m^\top,\qquad\mathcal I_{\bH\bH}=\tfrac12\bigl(S+\bar m\bar m^\top\bigr).
$$
By the Schur complement formula,
$$
\bigl[\mathcal I^{-1}\bigr]_{\alpha^2\alpha^2}=\Bigl(\mathcal I_{\alpha^2\alpha^2}-\mathcal I_{\alpha^2\bH}\,\mathcal I_{\bH\bH}^{-1}\,\mathcal I_{\bH\alpha^2}\Bigr)^{-1}=\frac{2\alpha^4}{1-\bar m^\top(S+\bar m\bar m^\top)^{-1}\bar m}.
$$
Moreover, the Sherman--Morrison formula gives $\bar m^\top(S+\bar m\bar m^\top)^{-1}\bar m=q/(1+q)$ with $q=\bar m^\top S^{-1}\bar m$. Therefore
$$
\bigl[\mathcal I^{-1}\bigr]_{\alpha^2\alpha^2}=2\alpha^4(1+q)=2\alpha^4\,\VIF(\bH),
$$
which proves the asymptotic normality with variance $2\alpha_0^4\VIF(\bH_0)$.

Next we compute the coverage of $J_\alpha(\hbH)$. By definition, $\alpha_0^2\in J_\alpha(\hbH)$ if and only if
$$
\chi^2_{N-1,\beta/2}\le\frac{N\hat\alpha^2(\hbH)}{\alpha_0^2}\le\chi^2_{N-1,1-\beta/2}.
$$
Since $\chi^2_{N-1,p}=N+z_p\sqrt{2N}+o(\sqrt N)$ and $z_{\beta/2}=-z_{1-\beta/2}$, this event can be written as
$$
-\sqrt2\,z_{1-\beta/2}+o(1)\le\frac{\sqrt N\bigl(\hat\alpha^2(\hbH)-\alpha_0^2\bigr)}{\alpha_0^2}\le\sqrt2\,z_{1-\beta/2}+o(1).
$$
By the first part, its probability tends to
$$
\Prob\Bigl(|\zeta|\le\frac{z_{1-\beta/2}}{\sqrt{\VIF(\bH_0)}}\Bigr)=2\Phi\Bigl(\frac{z_{1-\beta/2}}{\sqrt{\VIF(\bH_0)}}\Bigr)-1,\qquad\zeta\sim N(0,1).
$$
This limit is smaller than $1-\beta$ if and only if $\VIF(\bH_0)>1$. Since $S$ is positive definite, this happens if and only if $\bar m=\nabla\bar\kappa(\bH_0)\ne0$.

For the corrected interval, the delta method gives
$$
\sqrt N\bigl(\log\hat\alpha^2(\hbH)-\log\alpha_0^2\bigr)\dto N\bigl(0,2\VIF(\bH_0)\bigr).
$$
Also $\VIF(\hbH)\to\VIF(\bH_0)$ in probability, because $\VIF$ is continuous and $\hbH$ is consistent. Moreover,
$$
\alpha_0^2\in J^{\mathrm c}_\alpha\iff\bigl|\log\hat\alpha^2(\hbH)-\log\alpha_0^2\bigr|\le z_{1-\beta/2}\sqrt{2\VIF(\hbH)/N},
$$
so Slutsky's lemma gives $\Prob(\alpha_0^2\in J^{\mathrm c}_\alpha)\to1-\beta$.

Finally, for $m=1$ we have $\ell_1=\partial_H\log\bigl(h^{2H}f_H\bigr)=2\log h+\partial_H\log f_H$. Hence $\bar m=2\log h+\bar\ell_H$ and $S=v_H$, which gives \eqref{eq:vif1}.
\end{proof}

\begin{remark}\label{rem:vif}
(i) The quantity $2\pi\exp(\bar\kappa)$ is, by Kolmogorov's formula, the one-step prediction error variance of the sequence with spectral density $f(\cdot;\bH)$. In the parametrization $(\alpha^2e^{\bar\kappa(\bH)},\bH)$ the scale is orthogonal to the Hurst vector, but the scale $\alpha^2$ of the fBm components is not, and $\VIF-1$ measures the resulting confounding. It is the fixed-mesh counterpart of the confounding of scale and Hurst index under high-frequency observations, which forces a non-diagonal rate matrix in \cite{FukasawaTakabatake2019}.
(ii) In \eqref{eq:vif1} the mean $\bar\ell_H$ and the variance $v_H$ do not depend on $h$, so $\VIF$ grows like $4(\log h)^2/v_H$ as $h\to0$. Table~\ref{tab:vif} gives its values. For $H=0.7$, $\bar\ell_H=-1.554$ and $v_H=4.673$, so $\VIF=1.52$, $10.11$ and $35.05$ for $h=1$, $1/12$ and $1/252$, and the plug-in chi-square interval with nominal level $95\%$ has asymptotic coverage $88.8\%$, $46.2\%$ and $25.9\%$.
(iii) For $m=1$ the drift statistic $\widehat T$ does not depend on $h$: the concentrated likelihood depends on $h$ only through the factor $h^{2H}$, which is absorbed by $\alpha^2$. The scale, in contrast, is estimated through $\hat\alpha^2(\hat H)=\alpha_0^2h^{2(H_0-\hat H)}\times(\text{an $h$-free quantity})$, so the error of $\hat H$ is multiplied by $2|\log h|$.
\end{remark}

\begin{remark}\label{rem:tests2}
By the duality between confidence intervals and tests \cite{LehmannRomano2005}, Theorem~\ref{thm:plugdrift} shows that the plug-in Student test of the hypothesis $\theta=\theta_0$ has asymptotic size $\beta$, and Remark~\ref{rem:satt} suggests the quantile $t_{\hat\nu,1-\beta/2}$ in finite samples. For the scale, when the hypothesis $\alpha^2=\alpha_0^2$ is true, the plug-in chi-square test rejects it with asymptotic probability $2-2\Phi\bigl(z_{1-\beta/2}/\sqrt{\VIF(\bH_0)}\bigr)$, which exceeds $\beta$ whenever $\nabla\bar\kappa(\bH_0)\ne0$. By Table~\ref{tab:vif}, a nominal $5\%$ test for one component with $H=0.7$ and monthly data rejects a true hypothesis with probability about $54\%$. The test that rejects when $\alpha_0^2\notin J^{\mathrm c}_\alpha$ has asymptotic size $\beta$.
\end{remark}

\section{Simulation study}\label{sec:sim}

\subsection{Design}
All samples are simulated exactly, from the Cholesky factor of $C_N(\bH)$, with $\theta=\alpha=1$, and all intervals have nominal level $95\%$. The two-step estimator of Definition~\ref{def:est} is computed by bounded scalar minimization of the concentrated negative log-likelihood on $[0.505,0.995]$ when $m=1$, and, when $m=2$, by a $3\times3$ grid of starting values followed by bounded quasi-Newton minimization, the two indices being ordered afterwards. The information matrix \eqref{eq:info} and the factor $\VIF$ are computed by Gauss--Legendre quadrature with $600$ nodes after the substitution $\lambda=\pi u^4$, which handles the logarithmic singularities at the origin, and $\nabla_{\bH}\log V_N$ in Remark~\ref{rem:satt} by central differences with step $10^{-3}$. For $m=1$ the drift statistic and $\hat H$ do not depend on $h$ (Remark~\ref{rem:vif}(iii)), so each sample is simulated with $h=1$ and the scale intervals are evaluated for $h\in\{1,1/12,1/252\}$ through the exact relation $\hat\alpha^2_h=h^{2(H_0-\hat H)}\hat\alpha^2_1$. Python code reproducing all tables is available from the corresponding author.

\subsection{Misspecified Hurst vector}
Table~\ref{tab:misspec} uses $\tbH=\bH+\delta$ for one component with $H=0.7$ and $h=1$, and for two components with $\bH=(0.6,0.8)$ and $h=1/12$. In all eight cases the fitted exponent of $A_N$ differs from $2(\Hmax-\tHmax)=-2\delta$ by at most $0.0004$, so the lower bound of Theorem~\ref{thm:misspec}(a) is attained and the convergence in Theorem~\ref{thm:misspec}(b) has the same polynomial rate. Because the exponent $2|\delta|$ is small, the coverage of the drift interval moves away from the nominal level slowly and in a predictable direction: underestimating the largest index by $0.05$ lowers it to $81.4\%$ ($m=1$) and $80.3\%$ ($m=2$) at $N=2000$, and overestimating it by $0.05$ raises it to $99.8\%$ and $99.7\%$. The ratio $B_N$, on the other hand, varies by at most $0.002$ between $N=250$ and $N=2000$ and differs from $1$, so the chi-square interval collapses as predicted by Theorem~\ref{thm:misspec}(d). For two components an error of $0.02$ in both indices already reduces its coverage to $2.9\%$ or $2.4\%$ at $N=2000$.

\begin{table}[t]
\centering\small
\caption{Misspecified Hurst vector $\widetilde{\mathbf H}=\mathbf H+\delta$. Exact variance ratio $A_N$, fitted exponent of $A_N$ over $N\in\{250,500,1000,2000\}$ against the value $-2\delta$, exact $B_N$ ($N=2000$), and Monte Carlo coverage (\%) of the Student interval for $\theta$ and the chi-square interval for $\alpha^2$ (nominal $95\%$; $4000$ replications for $N\le1000$, $2000$ for $N=2000$).}
\label{tab:misspec}
\setlength{\tabcolsep}{4pt}
\begin{tabular}{rrrrrrrrrrr}
\toprule
& \multicolumn{2}{c}{$A_N$} & & & \multicolumn{3}{c}{Student, $\theta$} & \multicolumn{3}{c}{chi-square, $\alpha^2$}\\
\cmidrule(lr){2-3}\cmidrule(lr){6-8}\cmidrule(lr){9-11}
$\delta$ & $N=250$ & $2000$ & slope & $B_N$ & $250$ & $1000$ & $2000$ & $250$ & $1000$ & $2000$\\
\midrule
\multicolumn{11}{l}{$m=1$, $\mathbf H=0.7$, $h=1$}\\
-0.05 & 1.732 & 2.132 & +0.100 & 0.944 & 85.4 & 81.8 & 81.3 & 90.3 & 73.9 & 55.5\\
-0.02 & 1.245 & 1.353 & +0.040 & 0.972 & 91.8 & 90.5 & 90.4 & 94.1 & 90.4 & 86.1\\
+0.02 & 0.803 & 0.739 & -0.040 & 1.035 & 97.0 & 97.5 & 98.0 & 92.7 & 87.8 & 79.5\\
+0.05 & 0.578 & 0.470 & -0.100 & 1.106 & 99.3 & 99.6 & 99.8 & 78.3 & 37.0 & 11.9\\
\multicolumn{11}{l}{$m=2$, $\mathbf H=(0.6,0.8)$, $h=1/12$}\\
-0.05 & 1.352 & 1.665 & +0.100 & 0.746 & 85.9 & 83.5 & 80.3 & 9.5 & 0.0 & 0.0\\
-0.02 & 1.128 & 1.226 & +0.040 & 0.886 & 92.3 & 91.5 & 90.8 & 75.0 & 23.4 & 2.9\\
+0.02 & 0.886 & 0.815 & -0.040 & 1.135 & 96.9 & 98.0 & 97.8 & 68.6 & 18.0 & 2.4\\
+0.05 & 0.740 & 0.600 & -0.100 & 1.387 & 99.1 & 99.7 & 99.7 & 4.9 & 0.0 & 0.0\\
\bottomrule
\end{tabular}
\end{table}

\subsection{Estimated Hurst index, one component}
Table~\ref{tab:plugin1} reports the case $m=1$, $H=0.7$. The estimator $\hat H$ has a small downward bias, $-0.015$ at $N=250$ and $-0.004$ at $N=1000$, and its standard deviation is within $8\%$ of the Cram\'er--Rao value from Theorem~\ref{thm:mle}. For the drift, the plug-in Student interval undercovers ($90.6\%$ at $N=250$, $93.2\%$ at $N=1000$), in line with Theorem~\ref{thm:plugdrift} and the slow $N^{-1/2}\log N$ convergence of the estimated standard error. The correction of Remark~\ref{rem:satt} restores the level ($94.2\%$ to $95.1\%$), with $\hat\nu$ of order $N/(\log N)^2$ and a median increase of the interval width of $19\%$, $11\%$ and $6\%$ for $N=250$, $500$ and $1000$. For the scale, the coverage of the plug-in chi-square interval agrees with the limit $2\Phi(1.96/\sqrt{\VIF})-1$ of Theorem~\ref{thm:plugscale} already at $N=250$: about $88\%$ for $h=1$, $43$ to $47\%$ for monthly data and $22$ to $26\%$ for daily data. The corrected interval $J^{\mathrm c}_\alpha$ covers $89.7\%$ to $94.4\%$ and approaches the nominal level as $N$ grows; its remaining undercoverage at small $N$ comes from the bias of $\hat H$, which enters $\hat\alpha^2$ multiplied by $2|\log h|$.

\begin{table}[t]
\centering\small
\caption{One component, $H=0.7$, $\theta=\alpha=1$: Monte Carlo coverage (\%) of intervals with nominal level $95\%$ ($1000$ replications for $N=250,500$ and $500$ for $N=1000$), and summaries of $\hat H$. The last column gives the asymptotic coverage of the plug-in chi-square interval from Theorem~\ref{thm:plugscale}.}
\label{tab:plugin1}
\begin{tabular}{lrrrr}
\toprule
 & $N=250$ & $500$ & $1000$ & limit\\
\midrule
\multicolumn{5}{l}{Drift $\theta$ (all $h$)}\\
\quad oracle $I_\theta(\mathbf H_0)$ & 94.9 & 95.5 & 94.8 & 95.0\\
\quad plug-in $I_\theta(\hat{\mathbf H})$ & 90.6 & 93.1 & 93.2 & 95.0\\
\quad corrected, $t_{\hat\nu}$ & 94.2 & 95.1 & 94.4 & 95.0\\
\multicolumn{5}{l}{Scale $\alpha^2$}\\
\quad oracle $J_\alpha(\mathbf H_0)$, all $h$ & 94.2 & 93.7 & 94.8 & 95.0\\
\quad plug-in $J_\alpha(\hat{\mathbf H})$, $h=1$ & 88.6 & 88.6 & 88.2 & 88.8\\
\quad corrected $J^{\mathrm c}_\alpha$, $h=1$ & 92.8 & 93.4 & 94.4 & 95.0\\
\quad plug-in $J_\alpha(\hat{\mathbf H})$, $h=1/12$ & 43.3 & 47.3 & 44.2 & 46.2\\
\quad corrected $J^{\mathrm c}_\alpha$, $h=1/12$ & 89.7 & 90.8 & 93.0 & 95.0\\
\quad plug-in $J_\alpha(\hat{\mathbf H})$, $h=1/252$ & 22.2 & 25.7 & 25.0 & 25.9\\
\quad corrected $J^{\mathrm c}_\alpha$, $h=1/252$ & 90.1 & 91.3 & 93.8 & 95.0\\
\midrule
mean of $\hat H$ & 0.685 & 0.690 & 0.696 & 0.700\\
standard deviation of $\hat H$ & 0.0445 & 0.0302 & 0.0214 & \\
Cram\'er--Rao standard deviation & 0.0414 & 0.0293 & 0.0207 & \\
median of $\hat\nu$ & 7.5 & 12.0 & 19.7 & \\
\bottomrule
\end{tabular}
\end{table}

\subsection{Estimated Hurst indices, two components}
Table~\ref{tab:plugin2} reports the case $\bH=(0.6,0.8)$ with monthly data. By Table~\ref{tab:weakid}, the Cram\'er--Rao standard deviation of $\hat H_2$ is $0.20$ at $N=500$ and $0.14$ at $N=1000$, so the second index is weakly identified. Accordingly the two estimated indices merge in $64\%$ of the replications at $N=500$ and in $36\%$ at $N=1000$, and $\hat H_2$ is strongly biased towards $\hat H_1$. The plug-in drift interval undercovers ($83.5\%$ and $86.0\%$). The corrected interval overcovers, but it is more than twice as wide as the plug-in interval and its median $\hat\nu$ is about $2$; this is the signal, announced in Remark~\ref{rem:satt}, that the standard error of $\hat\theta$ is itself poorly determined. The plug-in chi-square interval covers about $36\%$, of the same order as the limit $29.1\%$ of Table~\ref{tab:vif}. The corrected scale interval reaches $88$ to $89\%$; it remains below the nominal level because $\VIF$ is evaluated at a biased $\hat{\bH}$. With two components and $N\le1000$, reliable inference therefore requires either longer series or a reduction of the model.

\begin{table}[t]
\centering\small
\caption{Two components, $\mathbf H=(0.6,0.8)$, $h=1/12$, $\theta=\alpha=1$: Monte Carlo coverage (\%) of intervals with nominal level $95\%$ ($400$ replications for $N=500$, $200$ for $N=1000$) and summaries of $\hat{\mathbf H}$. The row ``merged'' gives the percentage of replications with $\hat H_2-\hat H_1<0.01$.}
\label{tab:plugin2}
\begin{tabular}{lrr}
\toprule
 & $N=500$ & $1000$\\
\midrule
drift, oracle & 94.0 & 94.5\\
drift, plug-in & 83.5 & 86.0\\
drift, corrected $t_{\hat\nu}$ & 98.8 & 97.0\\
median width ratio, corrected/plug-in & 2.51 & 2.12\\
scale, oracle & 96.2 & 94.5\\
scale, plug-in & 37.2 & 35.5\\
scale, corrected $J^{\mathrm c}_\alpha$ & 89.2 & 88.5\\
\midrule
mean of $(\hat H_1,\hat H_2)$ & (0.628, 0.691) & (0.619, 0.725)\\
standard deviations & (0.041, 0.074) & (0.033, 0.079)\\
Cram\'er--Rao standard deviations & (0.035, 0.197) & (0.025, 0.139)\\
merged (\%) & 63.7 & 36.0\\
median of $\hat\nu$ & 1.76 & 2.08\\
\bottomrule
\end{tabular}
\end{table}

\subsection{Identifiability and the inflation factor}
Table~\ref{tab:weakid} illustrates Proposition~\ref{prop:info}. For monthly data the standard deviation of the larger index of a pair grows from $0.14$ to $0.26$ when the gap shrinks from $0.2$ to $0.05$, and for three indices it exceeds $0.4$ even with gaps of $0.15$ and $0.2$. For daily data the largest index of a triple has a standard deviation above $1$ at $N=1000$, so it cannot be estimated from such a sample. Table~\ref{tab:vif} shows that $\VIF$ increases with $H$ and with $|\log h|$; for two components and daily data it exceeds $500$.

\begin{table}[t]
\centering\small
\caption{Asymptotic Cram\'er--Rao standard deviations of $\hat H_1,\dots,\hat H_m$ for $N=1000$, computed from $\mathcal I(\psi)$ in \eqref{eq:info}; for other $N$ multiply by $(1000/N)^{1/2}$.}
\label{tab:weakid}
\begin{tabular}{lll}
\toprule
$\mathbf H$ & $h=1/12$ & $h=1/252$\\
\midrule
$(0.6, 0.8)$ & (0.025, 0.139) & (0.052, 0.486)\\
$(0.65, 0.75)$ & (0.064, 0.164) & (0.026, 0.295)\\
$(0.675, 0.725)$ & (0.160, 0.257) & (0.112, 0.343)\\
$(0.55, 0.7, 0.9)$ & (0.084, 0.191, 0.433) & (0.113, 0.739, 1.441)\\
$(0.65, 0.75, 0.85)$ & (0.196, 0.530, 0.576) & (0.117, 0.387, 1.213)\\
\bottomrule
\end{tabular}
\end{table}

\begin{table}[!htb]
\centering\small
\caption{Variance inflation factor $\mathcal V(\mathbf H)$ of Theorem~\ref{thm:plugscale} and, in parentheses, the asymptotic coverage (\%) of the plug-in chi-square interval with nominal level $95\%$.}
\label{tab:vif}
\begin{tabular}{lrrr}
\toprule
$\mathbf H$ & $h=1$ & $h=1/12$ & $h=1/252$\\
\midrule
$(0.6)$ & 1.08 (94.1) & 7.39 (52.9) & 28.92 (28.4)\\
$(0.7)$ & 1.52 (88.8) & 10.11 (46.2) & 35.05 (25.9)\\
$(0.8)$ & 3.45 (70.9) & 16.21 (37.4) & 46.70 (22.6)\\
$(0.9)$ & 17.12 (36.4) & 41.66 (23.9) & 86.92 (16.7)\\
$(0.6, 0.8)$ & 2.67 (77.0) & 27.68 (29.1) & 552.94 (6.6)\\
\bottomrule
\end{tabular}
\end{table}

\FloatBarrier
\section{Concluding remarks}\label{sec:concl}

With known Hurst indices, inference for the drift and the scale of a multi-mixed fBm is classical, and its exactness rests entirely on the indices being right. A wrong Hurst vector makes the drift interval fail at a polynomial rate governed by $\Hmax-\tHmax$, and makes the scale interval fail for every error. When the indices are estimated, the drift interval survives, with a finite-sample correction for the uncertainty of its standard error, but the scale interval does not, because $\alpha^2$ is confounded with the indices through $\log h$; the corrected interval $J^{\mathrm c}_\alpha$ restores the level. The inflation matters in practice. For example, Hurst indices between $0.82$ and $0.87$ are reported for four Nordic stock indices in \cite{DufitinemaPynnonenSottinen2022}, and for a single fBm with $H$ between $0.8$ and $0.9$ observed daily, Table~\ref{tab:vif} gives $\VIF$ between $46.7$ and $86.9$; a nominal $95\%$ plug-in interval for the scale then covers the true value with probability of only about $23\%$ to $17\%$. Clustered indices are weakly identified, so in applications the number of components should be kept small, and the separation of the estimated indices should be checked before the plug-in intervals are used.

Several questions remain open. Component-specific scales $\alpha_r$ lead to a covariance that is not a multiple of a known matrix, so the exact results of Proposition~\ref{prop:oracle} are lost, and we expect the identification of the scales to deteriorate as the indices approach each other, as in Proposition~\ref{prop:info}. Under infill asymptotics, with $Nh$ fixed and $h\to0$, the factor $\VIF$ grows without bound, which suggests the non-standard rates found for a single component in \cite{FukasawaTakabatake2019}. The number of components $m$ is also unknown in practice, and its selection is a model choice problem in a setting with weakly identified parameters. Finally, an application of the corrected intervals to financial or network data, of the kind analysed in \cite{DufitinemaPynnonenSottinen2022,SunChenLu2022,VidacsVirtamo1999}, is left for future work.

\section*{Declarations}

\subsection*{Funding}
This research received no external funding.

\subsection*{Competing interests}
The authors declare that they have no competing interests.

\subsection*{Data and code availability}
No external data were used. The numerical results in Section~\ref{sec:sim} come from the simulations described in the paper; the code is available from the corresponding author on reasonable request.


\begin{thebibliography}{99}

\bibitem{Adenstedt1974}
R.~K. Adenstedt, \emph{On large-sample estimation for the mean of a stationary random sequence}, Ann. Statist. \textbf{2} (1974), no.~6, 1095--1107.

\bibitem{AlajmiMliki2021}
S.~Alajmi and E.~Mliki, \emph{Mixed generalized fractional Brownian motion}, J. Stoch. Anal. \textbf{2} (2021), no.~2, Article~2.

\bibitem{Beran1994}
J.~Beran, \emph{Statistics for Long-Memory Processes}, Chapman \& Hall, New York, 1994.

\bibitem{BertinTorresTudor2011a}
K.~Bertin, S.~Torres, and C.~A. Tudor, \emph{Drift parameter estimation in fractional diffusions driven by perturbed random walks}, Statist. Probab. Lett. \textbf{81} (2011), no.~2, 243--249.

\bibitem{BertinTorresTudor2011b}
K.~Bertin, S.~Torres, and C.~A. Tudor, \emph{Maximum-likelihood estimators and random walks in long memory models}, Statistics \textbf{45} (2011), no.~4, 361--374.

\bibitem{BorosEtAl2024}
D.~Boros, B.~Csan\'ady, I.~Ivkovic, L.~Nagy, A.~Luk\'acs, and L.~M\'arkus, \emph{Deep learning the Hurst parameter of linear fractional processes and assessing its reliability}, Qual. Reliab. Eng. Int. \textbf{40} (2024), no.~8, 4228--4246.

\bibitem{CaiHuangSunXiao2022}
C.-H. Cai, Y.-Z. Huang, L.~Sun, and W.-L. Xiao, \emph{Maximum likelihood estimation for mixed fractional Vasicek processes}, Fractal Fract. \textbf{6} (2022), no.~1, Article~44.

\bibitem{ChaouchElMaroufyElOmari2023}
H.~Chaouch, H.~El Maroufy, and M.~El Omari, \emph{Statistical inference for models driven by $n$-th order fractional Brownian motion}, Theory Probab. Math. Statist. \textbf{108} (2023), 29--43.

\bibitem{Cheridito2001}
P.~Cheridito, \emph{Mixed fractional Brownian motion}, Bernoulli \textbf{7} (2001), no.~6, 913--934.

\bibitem{Dahlhaus1989}
R.~Dahlhaus, \emph{Efficient parameter estimation for self-similar processes}, Ann. Statist. \textbf{17} (1989), no.~4, 1749--1766.

\bibitem{Dahlhaus2006}
R.~Dahlhaus, \emph{Correction note: Efficient parameter estimation for self-similar processes}, Ann. Statist. \textbf{34} (2006), no.~2, 1045--1047.

\bibitem{DufitinemaPynnonenSottinen2022}
J.~Dufitinema, S.~Pynn\"onen, and T.~Sottinen, \emph{Maximum likelihood estimators from discrete data modeled by mixed fractional Brownian motion with application to the Nordic stock markets}, Comm. Statist. Simulation Comput. \textbf{51} (2022), no.~9, 5264--5287.

\bibitem{ElOmari2023}
M.~El Omari, \emph{Parameter estimation for $n$th-order mixed fractional Brownian motion with polynomial drift}, J. Korean Statist. Soc. \textbf{52} (2023), no.~2, 450--461.

\bibitem{ElOmari2025}
M.~El Omari, \emph{Statistical inference for $n$th-order mixed fractional Brownian motion with polynomial drift}, Mod. Stoch. Theory Appl. \textbf{12} (2025), no.~2, 169--187.

\bibitem{FukasawaTakabatake2019}
M.~Fukasawa and T.~Takabatake, \emph{Asymptotically efficient estimators for self-similar stationary Gaussian noises under high frequency observations}, Bernoulli \textbf{25} (2019), no.~3, 1870--1900.

\bibitem{HarbaDouziElHajji2012}
R.~Harba, H.~Douzi, and M.~El Hajji, \emph{Maximum likelihood estimation, interpolation and prediction for fractional Brownian motion}, in: Image and Signal Processing (ICISP 2012), Lecture Notes in Comput. Sci., vol.~7340, Springer, Berlin, 2012, pp.~326--332.

\bibitem{HuNualartXiaoZhang2011}
Y.~Hu, D.~Nualart, W.~Xiao, and W.~Zhang, \emph{Exact maximum likelihood estimator for drift fractional Brownian motion at discrete observation}, Acta Math. Sci. Ser. B \textbf{31} (2011), no.~5, 1851--1859.

\bibitem{KetratTanadkithirun2026}
P.~Ketrat and R.~Tanadkithirun, \emph{Parameter estimation for Brownian motion with trend mixed with two fractional Brownian motions}, Sci. Technol. Asia \textbf{31} (2026), no.~3, 14--41.

\bibitem{KuangLiu2015}
N.~Kuang and B.~Liu, \emph{Parameter estimations for the sub-fractional Brownian motion with drift at discrete observation}, Braz. J. Probab. Stat. \textbf{29} (2015), no.~4, 778--789.

\bibitem{KubiliusMishuraRalchenko2017}
K.~Kubilius, Y.~Mishura, and K.~Ralchenko, \emph{Parameter Estimation in Fractional Diffusion Models}, Bocconi \& Springer Series, vol.~8, Springer, Cham, 2017.

\bibitem{KukushEtAl2022}
A.~Kukush, S.~Lohvinenko, Y.~Mishura, and K.~Ralchenko, \emph{Two approaches to consistent estimation of parameters of mixed fractional Brownian motion with trend}, Stat. Inference Stoch. Process. \textbf{25} (2022), no.~1, 159--187.

\bibitem{LehmannCasella1998}
E.~L. Lehmann and G.~Casella, \emph{Theory of Point Estimation}, 2nd ed., Springer, New York, 1998.

\bibitem{LehmannRomano2005}
E.~L. Lehmann and J.~P. Romano, \emph{Testing Statistical Hypotheses}, 3rd ed., Springer, New York, 2005.

\bibitem{LiebermanRosemarinRousseau2012}
O.~Lieberman, R.~Rosemarin, and J.~Rousseau, \emph{Asymptotic theory for maximum likelihood estimation of the memory parameter in stationary Gaussian processes}, Econometric Theory \textbf{28} (2012), no.~2, 457--470.

\bibitem{MaleAlmaniSottinen2023}
H.~Maleki Almani and T.~Sottinen, \emph{Multi-mixed fractional Brownian motions and Ornstein--Uhlenbeck processes}, Mod. Stoch. Theory Appl. \textbf{10} (2023), no.~4, 343--366.

\bibitem{MandelbrotVanNess1968}
B.~B. Mandelbrot and J.~W. Van Ness, \emph{Fractional Brownian motions, fractional noises and applications}, SIAM Rev. \textbf{10} (1968), no.~4, 422--437.

\bibitem{Mishura2016}
Y.~Mishura, \emph{Maximum likelihood drift estimation for the mixing of two fractional Brownian motions}, in: Stochastic and Infinite Dimensional Analysis, Trends Math., Birkh\"auser/Springer, Cham, 2016, pp.~263--280.

\bibitem{MishuraRalchenkoShklyar2017}
Y.~Mishura, K.~Ralchenko, and S.~Shklyar, \emph{Maximum likelihood drift estimation for Gaussian process with stationary increments}, Austrian J. Stat. \textbf{46} (2017), no.~3--4, 67--78.

\bibitem{MishuraRalchenkoShklyar2018a}
Y.~Mishura, K.~Ralchenko, and S.~Shklyar, \emph{Maximum likelihood estimation for Gaussian process with nonlinear drift}, Nonlinear Anal. Model. Control \textbf{23} (2018), no.~1, 120--140.

\bibitem{MishuraRalchenkoShklyar2018b}
Y.~Mishura, K.~Ralchenko, and S.~Shklyar, \emph{Parameter estimation for Gaussian processes with application to the model with two independent fractional Brownian motions}, in: Stochastic Processes and Applications, Springer Proc. Math. Stat., vol.~271, Springer, Cham, 2018, pp.~123--146.

\bibitem{MishuraRalchenkoZhelezniak2024}
Y.~Mishura, K.~Ralchenko, and H.~Zhelezniak, \emph{Numerical approach to the drift parameter estimation in the model with two fractional Brownian motions}, Comm. Statist. Simulation Comput. \textbf{53} (2024), no.~7, 3206--3220.

\bibitem{MishuraVoronov2015}
Y.~Mishura and I.~Voronov, \emph{Construction of maximum likelihood estimator in the mixed fractional--fractional Brownian motion model with double long-range dependence}, Mod. Stoch. Theory Appl. \textbf{2} (2015), no.~2, 147--164.

\bibitem{NorrosValkeilaVirtamo1999}
I.~Norros, E.~Valkeila, and J.~Virtamo, \emph{An elementary approach to a Girsanov formula and other analytical results on fractional Brownian motions}, Bernoulli \textbf{5} (1999), no.~4, 571--587.

\bibitem{RalchenkoYakovliev2023}
K.~Ralchenko and M.~Yakovliev, \emph{Asymptotic normality of parameter estimators for mixed fractional Brownian motion with trend}, Austrian J. Stat. \textbf{52} (2023), 127--148.

\bibitem{RalchenkoYakovliev2024}
K.~Ralchenko and M.~Yakovliev, \emph{Parameter estimation for fractional mixed fractional Brownian motion based on discrete observations}, Mod. Stoch. Theory Appl. \textbf{11} (2024), no.~1, 1--29.

\bibitem{Satterthwaite1946}
F.~E. Satterthwaite, \emph{An approximate distribution of estimates of variance components}, Biometrics Bull. \textbf{2} (1946), no.~6, 110--114.

\bibitem{Stiburek2017}
D.~Stib\r{u}rek, \emph{Statistical inference on the drift parameter in fractional Brownian motion with a deterministic drift}, Comm. Statist. Theory Methods \textbf{46} (2017), no.~2, 892--905.

\bibitem{SunChenLu2022}
L.~Sun, J.~Chen, and X.~Lu, \emph{Parameter estimation for discretized geometric fractional Brownian motions with applications in Chinese financial markets}, Adv. Contin. Discrete Models \textbf{2022} (2022), Paper No.~69.

\bibitem{VidacsVirtamo1999}
A.~Vid\'acs and J.~T. Virtamo, \emph{Maximum likelihood estimation of the parameters of fractional Brownian traffic with geometrical sampling}, in: D.~H.~K. Tsang and P.~J. K\"uhn (eds.), Broadband Communications: Convergence of Network Technologies, Kluwer, Boston, 1999, pp.~51--62.

\bibitem{XiaoZhangZhang2011}
W.~Xiao, W.~Zhang, and X.~Zhang, \emph{Maximum-likelihood estimators in the mixed fractional Brownian motion}, Statistics \textbf{45} (2011), no.~1, 73--85.

\end{thebibliography}
\end{document}